\documentclass{amsart}

\usepackage{amssymb}
\usepackage{amsmath}
\usepackage{amsthm}
\usepackage[mathscr]{eucal}

\usepackage[active]{srcltx}

\newtheorem{thrm}{Theorem}[section]
\newtheorem{lemma}[thrm]{Lemma}
\newtheorem{prop}[thrm]{Proposition}

\newtheorem{main}{Main Theorem}

\theoremstyle{definition}
\newtheorem{defn}[thrm]{Definition}

\theoremstyle{remark}
\newtheorem{remark}[thrm]{Remark}

\numberwithin{equation}{section}

\newcommand{\dbar}{$\bar{\partial}$}
\newcommand{\mdbar}{\bar{\partial}}

\newcommand{\lre}{\mathscr{E}}
\newcommand{\lra}{\mathscr{A}}
\newcommand{\lrm}{\mathcal{M}}
\newcommand{\lrl}{\mathcal{L}}

\newcommand{\lrg}{\mathcal{G}}

\newcommand{\lrt}{\mathcal{T}}

\newcommand{\lrn}{\mathcal{N}}

  \newcommand{\lrh}{\mathscr{H}}
\begin{document}

\bibliographystyle{plain}

\title[Principal parts]{Principal parts of operators in the
\dbar-Neumann problem on strictly pseudoconvex non-smooth domains}
\author{Dariush Ehsani$^1$ and Ingo Lieb$^2$}


\thanks{The first author was partially supported by the Alexander von Humboldt Stiftung
and by the Max Planck Gesellschaft.}

\maketitle

\vskip .5\baselineskip
 {\em
\noindent $^1$ Humboldt-Universit\"{a}t, Institut f\"{u}r
Mathematik, 10099 Berlin
\newline
Email: {\small\texttt{dehsani.math@gmail.com}}
 \\ \newline
 $^2$ Universit\"{a}t Bonn, Mathematisches Institut, Endenicher Allee 60, 53115 Bonn, Deutschland
\newline
Email: {\small\texttt{ilieb@math.uni-bonn.de}}
  }

\section{Introduction}

Let $D$ be a bounded domain in $\mathbb{C}^n$.  We consider the
operator of the Cauchy-Riemann equations
\begin{equation*}
\mdbar: L^2_{0,q}(D) \rightarrow L^2_{0,q+1}(D),
\end{equation*}
where the $L^2$-spaces on $D$ are defined in terms of a Hermitian
metric given on $\mathbb{C}^n$, its Hilbert space adjoint,
\begin{equation*}
\mdbar^{\ast}: L^2_{0,q+1}(D) \rightarrow L^2_{0,q}(D),
\end{equation*}
and the associated complex Laplacian
\begin{equation*}
\square_q = \bar{\partial} \bar{\partial}^{\ast} + \bar{\partial}
\bar{\partial}^{\ast} :
 L^2_{0,q}(D) \rightarrow L^2_{0,q}(D),
\end{equation*}
whose domain of definition is singled out by the conditions
\begin{equation*}
u   \in Dom(\bar{\partial})\cap Dom(\bar{\partial}^{\ast}),
 \quad \mdbar u \in Dom(\bar{\partial}^{\ast}),
\quad \mdbar^{\ast} u \in Dom(\bar{\partial}).
\end{equation*}

The \dbar-Neumann problem asks to solve the equation
\begin{equation*}
\square u = f \qquad \mbox{for } f \perp ker(\square),
\end{equation*}
with $u\in Dom(\square)$.  We will study this problem on strictly
pseudoconvex domains which may have singularities at the boundary.
More precisely:
\begin{defn}
$D\subset\subset \mathbb{C}^n$ is a Henkin-Leiterer (HL) domain if
there is a strictly plurisubharmonic smooth function $r$ on a
neighborhood $U$ of the boundary $\partial D$ such that
\begin{equation*}
U\cap D=\{ \zeta\in U: r(\zeta)<0\}.
\end{equation*}
\end{defn}
 We shall make the additional assumption that $r$ is a Morse
function.  Then $\partial D = \{ \zeta: r(\zeta)=0\}$ and we may
assume that $r$ has finitely many critical points on the boundary,
and none on $U\setminus \partial D$.

Under these - and even more general - conditions the \dbar-Neumann
problem is solvable in the following sense: there is a linear
operator
\begin{equation*}
N:L^2_{0,q}(D) \rightarrow Dom(\square)
\end{equation*}
such that one has the orthogonal decomposition
\begin{align*}
L^2_{0,q}(D) =& ker (\square) \oplus \square N \left( L^2_{0,q} (D)\right)\\
=&ker (\square) \oplus \mdbar\mdbar^{\ast} N\left(L^2_{0,q}
(D)\right)
 \oplus \mdbar^{\ast}\mdbar N\left(L^2_{0,q} (D)\right).
\end{align*}
$N$  is called the \dbar-Neumann operator.  For $q>0$, one knows
that the harmonic space $ker(\square)$ is zero; this is no longer
true on more general manifolds.  For $q=0$, $ker(\square)$ is the
space of square integrable holomorphic functions.

We can now formulate our aim: to express the abstract operators,
$N$, $\mdbar N$, $\mdbar^{\ast} N$, as integral operators with
explicit (in terms of the defining function and the metric)
kernels.  The expression should be valid up to error terms which
have stronger smoothing properties than the explicit terms;
consequently, the boundedness properties of the above operators in
various function spaces ($L^p$-spaces, for instance) can be read
off the corresponding properties of the integral operators (which
have to be established, of course).

This program has been implemented in the case of smoothly bounded
domains by the work of many people - see \cite{LiMi} for
historical comments; we carry it over to the non-smooth HL case.

In order to state our results we now describe some needed
conventions and notations which will be kept fixed throughout this
paper. The metric on $\mathbb{C}^n$ will be chosen to coincide,
near the boundary of $D$, with the Levi form of $r$:
\begin{equation}
\label{met}
 ds^2 = \sum r_{i\bar{j}}(\zeta)d\zeta_i
d\overline{\zeta}_j.
\end{equation}
Any such metric is called a Levi metric; the \dbar-Neumann problem
is formulated in terms of this metric.

We set
\begin{equation*}
\label{gamdef}
 \gamma(\zeta)=|\partial r(\zeta)|,
\end{equation*}
where the length is measured by the metric in (\ref{met}).

For a double differential form $\mathcal{K}(\zeta,z)$ on $D\times
D$ we define the corresponding integral operator, $K$ by the
formula
\begin{equation*}
 \label{defop}
K f(z) = \int_{\zeta\in D} f(\zeta)\wedge \ast_{\zeta} \overline{
 \mathcal{K} (\zeta,z)}
\end{equation*}
and call $\mathcal{K}$ the kernel of $K$.  Here $\ast_{\zeta}$ is
the Hodge operator for the metric (\ref{met}), and $f$ is a
differential form.  If the types of $f$ and $\mathcal{K}$ do not
match, the integral is 0 by definition.  We finally set
\begin{equation}
\label{adker}
 \mathcal{K}^{\ast}(\zeta,z) = \overline{\mathcal{K}(z,\zeta)} ;
\end{equation}
in particular $\gamma^{\ast}(\zeta)=\gamma(z)$.

The first step in our program has already been done by the first
author in \cite{Eh10}:
\begin{thrm}
\label{lrethrm}
 Let $D\subset\subset X$ be a HL domain in a complex manifold $X$,
 given by a Morse defining
function $r$.
   There are
integral operators of type 1,
\begin{equation*}
{\bf T}_q:L^2_{0,q+1}(D)\rightarrow L^2_{0,q}(D)
\end{equation*}
such that for
 $f\in L^2_{0,q}(D)\cap
Dom(\bar{\partial})\cap Dom(\bar{\partial}^{\ast})$
\begin{equation*}
 f= {\bf T}_q\bar{\partial}f+ {\bf T}_{q-1}^{\ast}\bar{\partial}^{\ast}f
+\mbox{ error terms } \quad \mbox{ for } 1\le q< n=dim X,
\end{equation*}
where the error terms, after multiplication with suitable powers
of $\gamma$, involve only operators with better smoothing
properties than the principal terms, with $f$, $\mdbar f$, and
$\mdbar^{\ast}f$ arguments.
\end{thrm}
 The error terms will be explicitly described in Section
 \ref{repprin}, in Theorem \ref{bir}.
The type will be defined in Section \ref{admis}: it describes the
continuity properties of the operators in question.  A detailed
description of the operators $T_q$ is given \cite{Eh10}; we will
resume it in Section \ref{prelim}.  Although $X$ above can be an
arbitrary complex manifold we shall restrict attention, in this
paper, to $X=\mathbb{C}^n$.  The necessary adjustments in the
general case can be made as in \cite{LiMi} or \cite{LR87}.

In order to express our main results, we introduce the notion of
principal part.  We only indicate here what we mean and refer to
Section \ref{zop} for the precise definition.
 Let $A:L^2(D)\rightarrow L^2(D)$.
  $B$ is called a principal part of $A$ if for all large $L$
\begin{equation*}
\gamma^LA=\gamma^L B  + C,
\end{equation*}
 where now $\gamma^LB$ is an admissible integral operator (to be precise: a $Z$-
 operator) and
 where
 $C$ has better continuity properties than $\gamma^LB$.  We recall
the definition of "admissible" in Section \ref{admis}, the
definition of a $Z$-operator in Section \ref{zop}.

Our main results are the calculations of the principal parts of
the operators $N_q$, $\mdbar N_q$, and $\mdbar^{\ast} N_q$.  For
$N_q$ we calculate explicitly an integral operator, $N_q^0$, with
kernel, $\lrn_q$, such that we have
\begin{main}
\label{mainpp} For $1\le q\le n-3$:\\
 $a)$ $N_q^0$ is a principal part of the Neumann operator, $N_q$\\
 $b)$ ${\bold T}_{q-1}$ is a principal part of $\mdbar^{\ast}N_q$ \\
 $c)$ ${\bold T}_{q}^{\ast}$ is a principal part of $\mdbar N_q$.
\end{main}
Similar results hold for $q=n-2$, but will not be proven here.
See \cite{LiMi} for details.

We can interpret some of our earlier results on the Bergman
projector in terms of principal parts:
\begin{main}
The admissible operator $P_0^{\ast}$ of \cite{EhLi} is a principal
part of the Bergman projector $P$.
\end{main}
 See also \cite{Eh09c} for the Bergman projection in the setting
of domains in complex manifolds.

From the above representation we get, in view of the known
continuity properties of admissible operators, estimates for the
Neuman operator, which we express as follows:
\begin{main}
\label{mainnq} For $q$ as in Main Theorem \ref{mainpp}, and for
all $p\ge 2$ and $s$ such that
\begin{equation*}
\frac{1}{s} > \frac{1}{p} - \frac{1}{n+1}
\end{equation*}
we have the estimates
\begin{equation*}
a) \| \gamma^{3(n+2)} N_q f \|_{L^{s}} \lesssim \| \gamma^2 f
\|_{L^p} + \|f\|_{L^2},
\end{equation*}
and for $f\in dom\square \subset L^2_{0,q}(D)$
\begin{equation*}
\| \gamma^{3(n+2)}  f \|_{L^{s}} \lesssim \| \gamma^2 \square f
\|_{L^p} + \|f\|_{L^2} .
\end{equation*}
\end{main}

These $L^p$ estimates are of course only a typical example of the
use of our analysis of the Neumann operator; it is possible to
obtain estimates in other norms (see \cite{Eh09c} or \cite{LiMi});
they will all invoke the $\gamma$ weights.

Much of the work on this paper was done while the first author was
 visiting the Max Planck Institute for Mathematics in Bonn, and
 special gratitude is extended to the Institute for their
 invitation.  Moreover, both authors had the opportunity to
 cooperate at the Erwin Schr\"{o}dinger Institute in Vienna: we
 extend our sincere thanks to the Institute.

\section{Geometric data}
\label{geo}

We need the following data, all of which are given by the defining
function, $r$:
\begin{enumerate}
\item $\rho^2(\zeta,z)$, (square of) the geodesic distance for the
metric (\ref{met}),
 \item
 \begin{equation*}
 P(\zeta,z)=\rho^2(\zeta,z)+2\frac{r(\zeta)}{\gamma(\zeta)}
  \frac{r(z)}{\gamma(z)},
 \end{equation*}
the extended (squared) distance function,
 \item $F(\zeta,z)$, the Ram\'{i}rez-Henkin function of the domain
 (resp. of $r$).  Its definition and properties can be
 found in \cite{LiMi} or \cite{Ra}.  Let us only note that it is
 holomorphic in $z$, smooth in both variables, and that it gives
 rise to
\item
\begin{equation*}
\Phi(\zeta,z)=F(\zeta,z)-r(\zeta),
\end{equation*}
the extended Ram\'{i}rez-Henkin function, which satisfies the
crucial estimates
\begin{align*}
&\mbox{Re }\Phi(\zeta,z)>0 \mbox{ on } \partial D\times D,\\
&\Phi(\zeta,\zeta)=0 \quad \mbox{for } \zeta\in \partial D,\\
&|\Phi(\zeta,z)| \gtrsim \rho^2(\zeta,z) + |r(\zeta)| + |r(z)| +
|\mbox{Im }\Phi (\zeta,z)|.
\end{align*}
It is from these data that all the following integral kernels will
be constructed.
\end{enumerate}

\section{Admissible operators}
\label{admis}

 We write $\xi_k(\zeta)$ for a function with the
 property
\begin{equation*}
 |\gamma^{\alpha}D^{\alpha}_{\zeta}
 \xi_k(\zeta)|\lesssim \gamma^k.
\end{equation*}

 We shall write
 $\lre_j$ for those double forms on open sets
$U\subset \mathbb{C}^n\times \mathbb{C}^n$ such that $\lre_{j}$ is
smooth and satisfies
\begin{equation*}
\lre_{j}(\zeta,z)\lesssim  \rho^j(\zeta,z),
\end{equation*}
where $\rho$ coincides with the geodesic distance.
 In many cases we work with similar forms which are not necessarily
 smooth up to the boundary, and for such forms we define $\sigma_j$, $j\ge0$,for those double forms which
 are smooth on
open sets $U\subset D\times D$ such that
\begin{align*}
|\sigma_j| &\lesssim \lre_j\\
D_{\zeta} \sigma_j&=\xi_0 \sigma_{j-1}\\
D_{z} \sigma_j&=\xi_0^{\ast} \sigma_{j-1}
.
\end{align*}
Here and below $\xi_k^{\ast}=\overline{\xi}_k(z)$, the $\ast$
having a similar meaning for other functions of one variable.

\begin{defn} A double differential form $\lra(\zeta,z)$ on
$\overline{D}\times\overline{D}$ is an
\textit{admissible} kernel, if it has the following properties:
\begin{enumerate}
\item[i)] $\lra$ is continuous on $ \overline{D}\times
\overline{D}-\Lambda$, where $\Lambda$ is the boundary diagonal,
and smooth except possibly at the singular points $(\zeta,z)$ with
$\gamma(\zeta)=0$ or $\gamma(z)=0$.
 \item[ii)] For each point $(\zeta_0,\zeta_0)\in \Lambda$ there is
 a neighborhood $U\times U$ of $(\zeta_0,\zeta_0)$ on which $\lra$ or $\overline{\lra}$
 has the representation
 \begin{equation}
 \label{typerep}
  \xi_N\xi_M^{\ast} \sigma_{j}
  P^{-t_0}\Phi^{t_1}\overline{\Phi}^{t_2}
   \Phi^{\ast t_3}\overline{\Phi}^{\ast t_4} r^l r^{\ast m}
 \end{equation}
with $N,M,j, t_0, \ldots, m$ integers and $j,
t_0, l, m \ge 0$,
 $-t=t_1+\cdots+t_4\le 0$, $N, M\ge
0$, $l+m\le t+1$, and $N+M\ge 0$.
\end{enumerate}
We define the \textit{type} of $\lra(\zeta,z)$ to be
\begin{equation*}
\tau=2n+j+\min\{2, t-l-m,N+M\} -2(t_0+t-l-m).
\end{equation*}
\end{defn}
The type controls the regularity properties of the operator: the
larger the type, the better the regularity - see Proposition
\ref{zmap}.  Type 0 is at the edge of integrability - see
\cite{EhLi}; here we only work with positive type kernels.

Double forms $\lre_{j-2n}$ will be called isotropic kernels of
type $j$. Operators with the corresponding kernels will be called
isotropic (resp. admissible) of the corresponding type.

An important example of a type 2 isotropic kernel is
\begin{equation*}
\Gamma_{0q}=\Gamma_{0q}(\zeta,z),
\end{equation*}
a parametrix of the complex Laplacian; its derivatives
$\mdbar_{\zeta} \Gamma_{0q}$ and $\vartheta_{\zeta} \Gamma_{0q}$
are of type 1.  They are part of the $\lrt_q$-kernels and the
corresponding ${\bf T}_q$ operators mentioned above and defined in
the next paragraph.

In the next theorem and throughout this paper we shall denote the
kernels of the operators ${\bf T}_q$ by $\lrt_q$; the adjoint
operator has the kernel $\lrt_q^{\ast}$ - see (\ref{adker}).

We will use $\lra_l$ for a generic admissible kernel of type $l$
and $\lre_{l-2n}$ for a generic isotropic kernel of type $l$. More
elaborate and more general concepts (the "$Z$-operators") will be
defined in Section \ref{zop}.

Our main theorems follow from the following result whose proof
takes up the bulk of the present paper:
\begin{thrm}
\label{nkern} There are explicit kernels, $\lrn_q$, $0\le q\le
n-1$,
 which satisfy
\begin{align*}
& \lrn_q=\lrn_q^{\ast}+\frac{1}{\gamma}\lra_{3}
+\frac{1}{\gamma^{\ast}}\lra_{3}\\
&
\mdbar_{\zeta}\lrn_q=\lrt_q+\frac{1}{\gamma\gamma^{\ast}}\lra_{2}
 \\
&\left. \ast_{\zeta} \lrn_q\right|_{\partial D}=0.
\end{align*}
For $1\le q\le n-1$ we have
\begin{equation*}
\lrn_q= \frac{1}{\gamma^{\ast}}\lra_{2} + \Gamma_{0q}.
\end{equation*}
Moreover, if $1\le q\le n-2$, we have
\begin{equation*}
 \mdbar^{\ast}_{\zeta}\lrn_q=\lrt_{q-1}^{\ast}+\frac{1}{\gamma\gamma^{\ast}}\lra_{2}
.
\end{equation*}
The kernels $\lra_{2}$ above satisfy
\begin{equation*}
 \mdbar_{\zeta} \lra_{2}=\lra_{1}+\frac{1}{\gamma}\lra_{2}.
\end{equation*}
\end{thrm}

\section{Preliminary calculations}
\label{prelim}

In the next few lemmas we will often refer to a particular choice of local cooridinates.
We work in coordinate patch
near a boundary point of $D$ and define orthonormal frame of
$(1,0)$-forms on a neighborhood $U\cap D$ with
$\omega^1,\ldots,\omega^n$ where $\partial
r=\gamma\omega^n$ as the orthonormal frame,
and $L_1, \dots, L_n$ comprising the dual
frame. These operators refer to the variable $\zeta$.  When they
are to refer to the variable $z$, they will be denoted by
$\Theta^j$ and $\Lambda_j$, respectively.

We fix the point $\zeta$ and choose local coordinates
$z$ such that
\begin{equation}
\label{coorz}
 dz_j(\zeta)=\Theta_j(\zeta).
\end{equation}
 From the Morse Lemma,
near the critical points of $r$, denoted by $p_1,\ldots,p_k$, we can take
$\varepsilon$ small enough so that in each
\begin{equation*}
U_{2\varepsilon}(p_j)=\{\zeta: D\cap |\zeta-p_j|<2\varepsilon\},
\end{equation*}
for $j=1,\ldots,k$,
 there are coordinates $u_{j_1},\ldots,u_{j_m},v_{j_{m+1}},\ldots,v_{j_{2n}}$ such
 that
 \begin{equation}
 \label{rcoor}
 -r(\zeta)=u_{j_1}^2+\cdots+u_{j_m}^2-v_{j_{m+1}}^2-\cdots - v_{j_{2n}}^2,
 \end{equation}
with $u_{j_{\alpha}}(p_j)=v_{j_{\beta}}(p_j)=0$ for all $1\le
\alpha\le m$ and $m+1\le\beta\le 2n$.
Working in a neighborhood of a singularity in the boundary and
using such coordinates, we see
 first that for $j=1,\ldots,n$, $L_j$ is a sum
 of terms of the form $\xi_0\Lambda$, where $\Lambda$ here and below denotes any
 smooth first order differential
operator.  Similarly, $\Lambda_j$ is a sum
 of terms of the form $\xi_0^{\ast}\Lambda$.

Working now with the case $j=n$, the other cases being handled similarly, we see
  $\Lambda_n -\frac{\partial}{\partial z_n}$ is a sum of
  terms of the form
\begin{equation}
\label{derapp}
\left(\frac{a(z)}{\gamma(z)}-\frac{a(\zeta)}{\gamma(\zeta)}\right)
\Lambda =
(\xi_{-1}\xi_0^{\ast}\sigma_1)\Lambda,
\end{equation}
where $a$ is a smooth function such that $|a(\zeta)|\lesssim \gamma$.
 (\ref{derapp}) follows from
\begin{align*}
\frac{a(z)}{\gamma(z)}-\frac{a(\zeta)}{\gamma(\zeta)}&
 =a(z)\frac{\gamma(\zeta)-\gamma(z)}{\gamma(\zeta)\gamma(z)}+\frac{a(z)-a(\zeta)}{\gamma(\zeta)}\\
&=\xi_0^{\ast}\frac{\sigma_1}
 {\gamma(\zeta)}+\xi_{-1}\lre_1\\
&=\xi_{-1}\xi_0^{\ast}\sigma_1.
\end{align*}
We thus use repeatedly
\begin{equation}
\label{lamsym}
 \Lambda_j -\frac{\partial}{\partial z_j}=
 (\xi_{-1}\xi_0^{\ast}\sigma_1)\Lambda.
\end{equation}
By symmetry, we also have
\begin{equation*}
L_j -\frac{\partial}{\partial \zeta_j}=
 (\xi_{-1}\xi_0^{\ast}\sigma_1)\Lambda.
\end{equation*}

Coordinates taken as in (\ref{coorz}) give us the following
\begin{align*}
&\zeta_j=\xi_1\\
&\zeta_j-z_j=\sigma_1.
\end{align*}

We now collect various properties of functions comprising the integral kernels.
\begin{lemma}
 \label{phisymm}
\begin{align*}
&i.\ r(\zeta)=-\Phi(\zeta,z) + \lre_1\\
&ii.\ \Phi(\zeta,z)=\Phi^{\ast}(\zeta,z)+\lre_3.
\end{align*}
\end{lemma}
\begin{proof}
$i.$ follows as an immediate consequence of the definition of the function $\Phi$.
  $ii.$ follows as in the smooth case (see \cite{LiMi}).
\end{proof}

\begin{lemma}
 \label{lphi}
$i.$
\begin{align*}
\Lambda_n\Phi&=-\gamma+\xi_{0}\xi_0^{\ast}\sigma_1,\\
\overline{L}_n\Phi&=-\gamma^{\ast}
+\xi_{0}\xi_0^{\ast}\sigma_1.
\end{align*}
$ii.$ $\forall j$
\begin{align*}
&\overline{\Lambda}_j \Phi =\xi_0^{\ast}\lre_2,\\
&L_j\Phi=\xi_0\lre_2.
\end{align*}
$iii.$ For $j<n$
\begin{align*}
&\Lambda_j \Phi =\xi_0\xi_0^{\ast}\sigma_1,\\
&L_j\overline{\Phi}=\xi_{0}\xi_0^{\ast}\sigma_1.
\end{align*}
\end{lemma}
\begin{proof}
$i.$
 \begin{align*}
\Lambda_n \Phi =& \sum_{j=1}^n\frac{\partial r}{\partial \zeta_j}
 \Lambda_n(\zeta_j-z_j)+\xi_0^{\ast}\lre_{1}\\
=&\sum_{j<n}\frac{\partial r}{\partial \zeta_j}
 (\xi_{-1}\xi_0^{\ast}\sigma_1)\Lambda(\zeta_j-z_j)
  +\xi_0^{\ast}\lre_1
  \\
&+
\frac{\partial r}{\partial\zeta_n}
(-1+(\xi_{-1}\xi_0^{\ast}\sigma_1)\Lambda(\zeta_n-z_n))\\
=&-\gamma+\xi_{0}\xi_0^{\ast}\sigma_1.
\end{align*}

The second relation in $i.$ follows by taking conjugates, and by
Lemma \ref{phisymm}.\\
 $ii.$
 That $\overline{\Lambda}_j\Phi = \xi_0^{\ast}\lre_2$ is clear.
 $L_j\Phi=\xi_0\lre_2$ then follows by taking the adjoint and
using the fact that $\Phi -\Phi^{\ast}=\lre_3$. \\
$iii.$ As we wrote in the proof of $i$, we write
\begin{align*}
\Lambda_j \Phi &= \sum_{k=1}^n\frac{\partial r}{\partial \zeta_k}
 \Lambda_j(\zeta_k-z_k)+\xi_0^{\ast}\lre_{1}\\
=&\sum_{k\ne j}\frac{\partial r}{\partial \zeta_k}
 (\xi_{-1}\xi_0^{\ast}\sigma_1)\Lambda(\zeta_k-z_k)
 +\xi_0^{\ast}\lre_{1}\\
&+
\frac{\partial r}{\partial\zeta_j}
(-1+(\xi_{-1}\xi_0^{\ast}\sigma_1)\Lambda(\zeta_j-z_j))\\
=&\xi_{0}\xi_0^{\ast}\sigma_1.
\end{align*}
$L_j\overline{\Phi}=\xi_{0}\xi_0^{\ast}\sigma_1$ follows similarly.
\end{proof}

\begin{lemma}
 \label{lp}
$i.$
\begin{align*}
 \gamma\Lambda_n P=&
 -2 \overline{\Phi}
  +
 \xi_{1}\xi_{-1}^{\ast}(P+\lre_2)+\xi_0\xi_0^{\ast}\sigma_2\\
 \gamma^{\ast}L_n P=&
 -2 \Phi^{\ast}
  +\xi_{-1}\xi_1^{\ast}(P+\lre_2)+\xi_0\xi_0^{\ast}\sigma_2.
\end{align*}
$ii.$  For $j<n$
\begin{align*}
L_j\overline{\Phi}=&L_j\rho^2+\xi_{-1}\xi_0^{\ast}\sigma_2\\
=&L_jP+\xi_{0}\xi_{-1}^{\ast}r^{\ast}+\xi_{-1}\xi_{0}^{\ast}\sigma_2.
\end{align*}
$iii.$
\begin{align*}
\gamma\gamma^{\ast}\left(2P-\sum_{j<n}|L_j\rho^2|^2\right)=&
 4|\Phi|^2+r\xi_0\xi_0^{\ast}\sigma_2+\xi_0\xi_1^{\ast}\sigma_3
 +\xi_0\xi_0^{\ast}\sigma_4.
\end{align*}
\end{lemma}
\begin{proof}
Variants of $i.$ and $iii.$ were proved in \cite{Eh10}, and we will follow those proofs here.  \\
$i.$  We prove the first relation in $i.$, the second being a consequence of the first.
We have
\begin{equation*}
\Lambda_n P=\Lambda_n \rho^2
 +2\frac{r}{\gamma}-
  \frac{1}{\gamma^{\ast}}\xi_{0}(z)\frac{r r^{\ast}}{\gamma\gamma^{\ast}}
  .
\end{equation*}

With $\zeta$ fixed, we choose coordinates $z_j$, as in
(\ref{coorz}), so that
 $dz_j(\zeta)=\theta_j(\zeta)$, and we let
\begin{equation*}
R^2(\zeta,z)=\sum g_{jk}(\zeta)(\zeta_j-z_j)(\overline{\zeta}_k-\overline{z}_k),
\end{equation*}
where the $g_{jk}$ are determined by the metric, $ds^2=\sum g_{jk}d z_j d\overline{z}_k$.

With the metric chosen as the Levi metric, we write
\begin{equation*}
 g_{jk}(\zeta)=g_{jk}(z)+\sigma_1.
\end{equation*}
This gives us the relation
\begin{equation*}
\rho^2=R^2+\sigma_3,
\end{equation*}
and thus
\begin{align*}
\Lambda_n \rho^2&=\frac{\partial}{\partial z_n}R^2 +
(\xi_{-1}\xi_0^{\ast}\sigma_1)(\Lambda R^2)+\xi_0^{\ast}\sigma_2\\
&=\frac{\partial}{\partial z_n}R^2 +
(\xi_{-1}\xi_0^{\ast}\sigma_1)(\xi_0^{\ast}\sigma_1)+\xi_0^{\ast}\sigma_2\\
&=-2(\overline{\zeta}_n-\overline{z}_n)
+ \xi_{-1}\xi_0^{\ast}\sigma_2
,
\end{align*}
where the last line follows from $g_{jk}(\zeta)=2\delta_{jk}$ due
to the orthonormality of the $\Theta_j$.

Finally, this gives
\begin{align}
\nonumber
 \Lambda_n
P&=-2(\overline{\zeta}_n-\overline{z}_n)
 +2\frac{r}{\gamma}+\xi_{-1}^{\ast}\frac{rr^{\ast}}{\gamma\gamma^{\ast}}+ \xi_{-1}\xi_0^{\ast}\sigma_2\\
 \label{lnp}
 &=-2(\overline{\zeta}_n-\overline{z}_n)
 +2\frac{r}{\gamma}+\xi_{-1}^{\ast}(P+\lre_2)+ \xi_{-1}\xi_0^{\ast}\sigma_2
 ,
\end{align}
where we use
\begin{equation*}
\frac{rr^{\ast}}{\gamma\gamma^{\ast}}=P+\lre_2
\end{equation*}
in the last line.

We compare (\ref{lnp}) to $\overline{\Phi}$ by
calculating the Levi polynomial, $F(\zeta,z)$ in the
above coordinates:
\begin{align}
\nonumber
 \overline{\Phi}(\zeta,z)
 &=\overline{F}(\zeta,z) - r(\zeta) +\sigma_{2}\\
\label{phicoor}
 &=\gamma(\zeta)(\overline{\zeta}_n-\overline{z}_n)-r(\zeta)
  +\xi_{0}\xi_0^{\ast}\sigma_2.
\end{align}

$ii.$
Again we use
 \begin{equation*}
 \rho^2=R^2+\sigma_3
 \end{equation*}
 below to obtain
 \begin{align*}
L_j\overline{\Phi}=&
 \sum_k \left(L_j\frac{\partial r}{\partial\overline{\zeta}_k}\right) (\overline{\zeta_k-z_k})+\xi_0\sigma_2\\
=& \sum_k \left[ \left(\frac{\partial}{\partial \zeta_j}+
(\xi_{-1}\xi_0^{\ast}\sigma_1)\Lambda\right)
 \frac{\partial r}{\partial\overline{\zeta}_k}\right] (\overline{\zeta_k-z_k}) +\xi_0\sigma_2\\
=& \sum_k\frac{\partial^2r}{\partial\zeta_j\overline{\zeta}_k}(\overline{\zeta_k-z_k})
+\xi_{-1}\xi_0^{\ast}\sigma_2 \\
=&\frac{\partial}{\partial \zeta_j}
 \left(\sum_{k,l}\frac{\partial^2r}{\partial\zeta_l\overline{\zeta}_k}(\zeta_l-z_l)(\overline{\zeta_k-z_k})
 \right)
+\xi_{-1}\xi_0^{\ast}\sigma_2\\
=& \frac{\partial}{\partial \zeta_j}
 \left(\rho^2+\sigma_3\right) +\xi_{-1}\xi_0^{\ast}\sigma_2\\
=& \left(L_j+(\xi_{-1}\xi_0^{\ast}\sigma_1)\Lambda
\right)\rho^2+\xi_{-1}\xi_0^{\ast}\sigma_2 \\
=&L_j\rho^2+\xi_{-1}\xi_0^{\ast}\sigma_2.
\end{align*}
We now use the relation
 \begin{align*}
 L_jP=&L_j\rho^2+\xi_{-1}\frac{rr^{\ast}}{\gamma\gamma^{\ast}}\\
=&L_j\rho^2+\xi_{0}\xi_{-1}^{\ast}r^{\ast}+\xi_{-1}\xi_{0}^{\ast}\sigma_2
\end{align*}
to
 finish the proof of $ii.$
\\
$iii.$
We have
\begin{align*}
|L_j\rho^2|^2 =&
 (L_j\rho^2)\left(\frac{\partial}{\partial \overline{\zeta}_j}+(
\xi_{-1}\xi_0^{\ast}\sigma_1)\Lambda\right)\rho^2\\
=&(L_j\rho^2)\frac{\partial\rho^2}{\partial \overline{\zeta}_j}
 +\xi_{-1}\xi_0^{\ast}\sigma_3\\
=&\left| \frac{\partial}{\partial
\zeta_j}\rho^2
\right|^2 +\xi_{-1}\xi_0^{\ast}\sigma_3\\
 =&4|\zeta_j-z_j|^2+\xi_{-1}\xi_0^{\ast}\sigma_3\\
=&4|\zeta_j-z_j|^2
+\xi_{-1}\xi_0^{\ast}\sigma_3
.
\end{align*}
We can then, with the relation
\begin{equation*}
P=2\sum_{j}|\zeta_j-z_j|^2+\sigma_3+2\frac{rr^{\ast}}{\gamma\gamma^{\ast}},
\end{equation*}
 write
\begin{equation*}
2P -\sum_{j<n} |L_j\rho^2|^2=
 4|\zeta_n-z_n|^2+4\frac{rr^{\ast}}{\gamma\gamma^{\ast}}
+\xi_{-1}\xi_0^{\ast}\sigma_3.
\end{equation*}

The calculations in (\ref{phicoor})
also give
\begin{equation*}
\Phi=\gamma(\zeta_n-z_n)-r+\xi_{0}\xi_0^{\ast}\sigma_2,
\end{equation*}
which we use in
\begin{align*}
\Phi\overline{\Phi}
 =&(\gamma(\zeta_n-z_n)-r(\zeta)+\xi_{0}\xi_0^{\ast}\sigma_2)\overline{\Phi}\\
 =&\gamma(\zeta_n-z_n)[\gamma(\overline{\zeta}_n-\overline{z}_n)
  -r(\zeta)+\xi_{0}\xi_0^{\ast}\sigma_2]-r(\zeta)\overline{\Phi}+
  (\xi_{0}\xi_0^{\ast}\sigma_2)\overline{\Phi}\\
 =&\gamma\gamma^{\ast}|\zeta_n-z_n|^2 -
 r(\zeta)[\gamma(\zeta_n-z_n)+\overline{\Phi}]+r(\zeta)\xi_0\xi_0^{\ast}\sigma_2+
 \xi_{1}\xi_0^{\ast}\sigma_3+\xi_{0}\xi_0^{\ast}\sigma_4,
\end{align*}
where we use $\gamma(\zeta)=\gamma(z)+\sigma_1$ and
$\Phi=\xi_1\sigma_1+\sigma_2-r$ in the last
step.

From Lemma \ref{phisymm} we have
\begin{align*}
\gamma(\zeta_n-z_n)+\overline{\Phi}
 &=\gamma(\zeta_n-z_n)+\overline{\Phi}^{\ast}
 +\lre_3\\
&=\gamma(\zeta_n-z_n)+\gamma^{\ast}(z_n-\zeta_n)-r(z)+\xi_0\xi_0^{\ast}\sigma_2\\
&=-r(z)+\xi_{0}\xi_0^{\ast}\sigma_2,
\end{align*}
and so we can write
\begin{equation*}
\Phi\overline{\Phi} =
\gamma\gamma^{\ast}|\zeta_n-z_n|^2
+rr^{\ast}+r\xi_0\xi_0^{\ast}\sigma_2+
\xi_1\xi_0^{\ast}\sigma_3+
\xi_0\xi_0^{\ast}\sigma_4
 .
\end{equation*}
iii. now easily follows.

\end{proof}

We want to compute the principal parts of the kernels $\lrt_q$
occurring in the integral representation \ref{lrethrm}. From
\cite{Eh10} we have
\begin{align}
\label{tqdef} &
\lrt_q^{}=\vartheta_{\zeta}\lrl_q^{}-\partial_z\lrl_{q-1}^{}+
\overline{\partial}_{\zeta}\Gamma_{0q}^{},
 \qquad q\ge 1\\
 \nonumber
& \lrt_0^{}= \vartheta_{\zeta}\lrl_0^{} -\ast\overline{K}_0^{}
 +\overline{\partial}_{\zeta}\Gamma_{00}^{},
\end{align}
where the various kernels are defined below.

 We start with the differential forms
\begin{align*}
&\beta(\zeta,z)=
 \frac{\partial_{\zeta}\rho^2(\zeta,z)}{\rho^2(\zeta,z)}\\
&\alpha_{}(\zeta,z)=\xi(\zeta)\frac{\partial
r_{}(\zeta)}{\phi_{}(\zeta,z)},
\end{align*}
where $\xi(\zeta)$ is a smooth patching function which is
equivalently 1 for $|r(\zeta)|<\delta$ and 0 for
$|r(\zeta)|>\frac{3}{2} \delta$, and $\delta>0$ is sufficiently
small.  We define
\begin{equation*}
C_q^{}=C_q(\alpha_{},\beta)
 = \sum_{\mu=0}^{n-q-2}\sum_{\nu=0}^{q}
 a_{q\mu\nu}C_{q\mu\nu}(\alpha_{},\beta),
\end{equation*}
where
\begin{equation*}
a_{q\mu\nu}=\left(\frac{1}{2\pi
i}\right)^n\binom{\mu+\nu}{\mu}\binom{n-2-\mu-\nu}{q-\mu}
\end{equation*}
and
\begin{equation*}
C_{q\mu\nu}(\alpha_{},\beta)
 = \alpha_{}\wedge
 \beta\wedge(\mdbar_{\zeta}\alpha_{})^{\mu}\wedge
 (\mdbar_{\zeta}\beta)^{n-q-\mu-2}\wedge(\mdbar_z\alpha_{})^{\nu}
 \wedge(\mdbar_z\beta)^{q-\nu}.
\end{equation*}
Denoting the Hodge $\ast$-operator by $\ast$, we then define
\begin{equation}
\label{lqc}
\lrl_q^{}(\zeta,z)=(-1)^{q+1}\ast_{\zeta}\overline{C_q^{}(\zeta,z)}.
\end{equation}
We also write
\begin{equation*}
K_q^{}(\zeta,z) = (-1)^{q(q-1)/2}\binom{n-1}{q}\frac{1}{(2\pi
i)^n}\alpha_{}\wedge(\mdbar_{\zeta}\alpha_{})^{n-q-1}
\wedge(\mdbar_z\alpha_{})^q
\end{equation*}
and
\begin{equation*}
\Gamma_{0,q}^{}(\zeta,z)
 =\frac{(n-2)!}{2\pi^n}\frac{1}{\rho^{2n-2}}\left(
 \mdbar_{\zeta}\partial_z\rho^2\right)^q.
\end{equation*}

For ease of notation we will drop here the superscripts
$\epsilon$, which were used in \cite{Eh10} to do calculations on
the smooth subdomains, $D_{\epsilon}=\{ r<-\epsilon \}$ noting
that the following calculations also hold when the kernels on
$D\times D$ are replaced with the corresponding kernels on
$D_{\epsilon}\times D_{\epsilon}$.  All formulas remain the same
and make sense when one looks at the appropriate weighted $L^p$
spaces.

\begin{lemma}
\label{lemmalq}
 The kernels $\lrl_q$ given in (\ref{lqc}) can be represented, for
$0\le q\le n-2$, in the following form:
\begin{equation}
\label{lq}
 \lrl_q= c_{nq}\sum_{{{0\le\mu\le n-q-2}\atop{j<n}}\atop{|L|=q}}
 {n-2-\mu \choose q} \frac{\overline{L}_j\rho^2}{\overline{\Phi}^{\mu+1}P^{n-\mu-1}}
\gamma
\overline{\omega}^{njL}\wedge\Theta^L + \lra_{3},
\end{equation}
where
\begin{equation*}
c_{nq}=2^{n-2}\left(\frac{1}{2\pi}\right)^n q!(n-q-2)!
\end{equation*}
and the terms $\lra_{3}$ satisfy
\begin{equation*}
\vartheta_{\zeta} \lra_{3}=\lra_{2},\
\partial_z\lra_{3}=\lra_{2},
\end{equation*}
\begin{equation*}
 \mdbar_{\zeta}\vartheta_{\zeta}\lra_{3}=\lra_{1}+\frac{1}{\gamma}\lra_2
 ,\
\mdbar_{\zeta}\partial_z\lra_{3}= \lra_{1}.
\end{equation*}
Alternatively, we can use
\begin{equation*}
 \mdbar_{\zeta}\vartheta_{\zeta}
 \lra_{3}=\lra_{1}+\frac{1}{\gamma^{\ast}}\lra_2
 .
\end{equation*}
\end{lemma}
\begin{proof}
(\ref{lq}) is given in \cite{Eh10}.  To see the error terms have
the property we write from \cite{Eh10}
\begin{equation*}
\lrl_q= \sum_{\mu=0}^{n-q-2}
\left( g_{q\mu}C_{q\mu}+\frac{\lre_2\wedge\mdbar r}
 {\overline{\Phi}^{\mu+1}P^{n-\mu-1}}\right) +\lre_0,
\end{equation*}
where
\begin{equation*}
g_{q\mu} = c_{nq} {n-2-\mu \choose q}
\end{equation*}
and
\begin{equation*}
C_{q\mu}=\sum_{{j<n}\atop{|L|=q}}
 \frac{\overline{L}_j\rho^2}{\overline{\Phi}^{\mu+1}P^{n-\mu-1}}
\gamma
\overline{\omega}^{njL}\wedge\Theta^L.
\end{equation*}

We have
\begin{align}
\label{oneder} \vartheta_{\zeta}
 \left(\frac{\lre_2\wedge \mdbar
r}{\overline{\Phi}^{\mu+1}P^{n-\mu-1}}\right)
 =&
 \frac{\lre_{2}+\xi_1\lre_{1}}{\overline{\Phi}^{\mu+1}P^{n-\mu-1}}
+ \frac{\lre_2\wedge \mdbar r}{\overline{\Phi}^{\mu+2}P^{n-\mu-1}}
(\lre_1)\\
\nonumber
&+\frac{\lre_2\wedge \mdbar r}{\overline{\Phi}^{\mu+1}P^{n-\mu}}\left(\lre_{1}+
 \xi_{0}\frac{r^{\ast}}{\gamma^{\ast}}\right)\\
\nonumber =&\lra_{2}.
\end{align}
Similarly,
\begin{equation*}
 \partial_z \left(\frac{\lre_2\wedge \mdbar r}{\overline{\Phi}^{\mu+1}P^{n-\mu-1}}\right)
 = \lra_{2}.
\end{equation*}

We use (\ref{oneder}) and Lemma \ref{lphi} to calculate
\begin{align*}
\mdbar_{\zeta}\vartheta_{\zeta} \left(\frac{\lre_2\wedge \mdbar
r}{\overline{\Phi}^{\mu+1}P^{n-\mu-1}}\right).
\end{align*}
We note $\mdbar\vartheta\mdbar r=0$, and calculate
\begin{align*}
\mdbar_{\zeta}\left(\frac{\vartheta_{\zeta}(\lre_2\wedge \mdbar
r)}{\overline{\Phi}^{\mu+1}P^{n-\mu-1}} \right)
=&\frac{\lre_{1}+\xi_1\lre_{0}}{\overline{\Phi}^{\mu+1}P^{n-\mu-1}}
+\frac{\lre_{2}+\xi_1\lre_{1}}
 {\overline{\Phi}^{\mu+2}P^{n-\mu-1}}\lre_2\\
&+\frac{\lre_{2}+\xi_1\lre_{1}}
 {\overline{\Phi}^{\mu+1}P^{n-\mu}}
 (\lre_{1}+
 \xi_{0}\xi_1^{\ast})\\
=& \lra_{1}.
\end{align*}

We also have
\begin{align*}
\mdbar_{\zeta} \Bigg(
 \frac{\lre_3\wedge \mdbar r}{\overline{\Phi}^{\mu+2}P^{n-\mu-1}}
\Bigg)
=&\frac{\xi_1\lre_{2}}{\overline{\Phi}^{\mu+2}P^{n-\mu-1}}
+ \frac{\lre_3\wedge \mdbar r}{\overline{\Phi}^{\mu+3}P^{n-\mu-1}}
(\lre_2)\\
\nonumber
&+\frac{\lre_3\wedge \mdbar r}{\overline{\Phi}^{\mu+2}P^{n-\mu}}(\lre_{1}+
 \xi_{0}\xi_1^{\ast})\\
\nonumber =&\lra_{1},
\end{align*}
and
\begin{align}
\label{newerror}
 \mdbar_{\zeta}\Bigg(
 \frac{\lre_2\wedge \mdbar r}{\overline{\Phi}^{\mu+1}P^{n-\mu}}&\left(\lre_{1}+
 \xi_{0}\frac{r^{\ast}}{\gamma^{\ast}}\right)\Bigg)\\
 \nonumber
=&\mdbar_{\zeta}\Bigg(
 \frac{\lre_2\wedge \mdbar r}{\overline{\Phi}^{\mu+1}P^{n-\mu}}\Bigg)
 \left(\lre_{1}+
 \xi_{0}\frac{r^{\ast}}{\gamma^{\ast}}\right)+
\frac{\lre_2\wedge \mdbar r}{\overline{\Phi}^{\mu+1}P^{n-\mu}}
(\lre_0+\xi_{-1}\xi_1^{\ast})
\\
\nonumber
 =&
\frac{\xi_1\lre_{1}}{\overline{\Phi}^{\mu+1}P^{n-\mu}}
\left(\lre_{1}+
 \xi_{0}\frac{r^{\ast}}{\gamma^{\ast}}\right)
+ \frac{\lre_2\wedge \mdbar r}{\overline{\Phi}^{\mu+2}P^{n-\mu}}
(\lre_2)(\lre_{1}+
 \xi_0\xi_1^{\ast})\\
 \nonumber
&+\frac{\lre_2\wedge \mdbar r}{\overline{\Phi}^{\mu+1}P^{n-\mu+1}}\left(\lre_{1}+
 \xi_{0}\frac{r^{\ast}}{\gamma^{\ast}}\right)(\mdbar_{\zeta}P)
 +\lra_1.
\end{align}

We can now write
\begin{equation*}
\lre_{1}+
 \xi_{0}\frac{r^{\ast}}{\gamma^{\ast}}
  = \sigma_1 +\xi_{0}\frac{r}{\gamma}
\end{equation*}
in (\ref{newerror}) and
\begin{align*}
\mdbar_{\zeta}P &= \lre_{1}+
 \xi_{0}\frac{r^{\ast}}{\gamma^{\ast}}\\
 &=\sigma_1+ \xi_{0}\frac{r}{\gamma}.
\end{align*}
We then use
\begin{align*}
\left(\frac{r}{\gamma}\right)^2
 &=\frac{rr^{\ast}}{\gamma\gamma^{\ast}}
 +\sigma_1\frac{r}{\gamma}\\
 &=P+\lre_2+\sigma_1\frac{r}{\gamma}
\end{align*}
to write
\begin{align*}
\frac{\lre_2\wedge \mdbar r}{\overline{\Phi}^{\mu+1}P^{n-\mu+1}}
\left(\xi_{0}\frac{r}{\gamma}\right)^2
 &=\frac{\lre_2\wedge \mdbar r}{\overline{\Phi}^{\mu+1}P^{n-\mu+1}}
  (P+\lre_2)+
\frac{\lre_2\wedge \mdbar r}{\overline{\Phi}^{\mu+1}P^{n-\mu+1}}
\sigma_1\frac{r}{\gamma}
  \\
  &=\lra_1 + \frac{1}{\gamma}\lra_2.
\end{align*}
Thus (\ref{newerror}) gives terms of the form
\begin{equation*}
\lra_1+\frac{1}{\gamma}\lra_2.
\end{equation*}
Alternatively, using
\begin{equation*}
\mdbar_xP = \lre_{1}+
 \xi_{0}\frac{r^{\ast}}{\gamma^{\ast}}
\end{equation*}
directly in (\ref{newerror}) leads to terms of the form
\begin{equation*}
\lra_1+\frac{1}{\gamma^{\ast}}\lra_2.
\end{equation*}

Similar calculations hold for $\mdbar_{\zeta}\partial_z\lra_{3}$.
\end{proof}

In order to calculate the derivations of $\lrl_q$ which turn up in
our formula (\ref{tqdef}), we set
\begin{align*}
\lrm_{kj}^{\mu}&=\Lambda_k\left(
 \frac{\overline{L}_j\rho^2}{\overline{\Phi}^{\mu+1}P^{n-\mu-1}}\right)\\
&=-(\mu+1)\frac{\overline{L}_j\rho^2}{\overline{\Phi}^{\mu+2}P^{n-\mu-1}}
  \Lambda_k\overline{\Phi}+\frac{1}{\overline{\Phi}^{\mu+1}}\Lambda_k\left(
 \frac{\overline{L}_j\rho^2}{P^{n-\mu-1}}\right)\\
\widetilde{\lrm}_{kj}^{\mu}&=L_k\left(
 \frac{\overline{L}_j\rho^2}{\overline{\Phi}^{\mu+1}P^{n-\mu-1}}\right)\\
&=-(\mu+1)\frac{\overline{L}_j\rho^2}{\overline{\Phi}^{\mu+2}P^{n-\mu-1}}
  L_k\overline{\Phi}+\frac{1}{\overline{\Phi}^{\mu+1}}L_k\left(
 \frac{\overline{L}_j\rho^2}{P^{n-\mu-1}}\right).
\end{align*}
From
\begin{equation*}
\Lambda_k\overline{L}_j \rho^2
 = -2\delta_{jk} +\xi_{-1}\xi_0^{\ast}\sigma_1
  +\xi_{-1}\xi_{-1}^{\ast}\sigma_2
 \end{equation*}
 we have, for $k<n$,
\begin{equation*}
\frac{1}{\overline{\Phi}^{\mu+1}}\Lambda_k\left(
 \frac{\overline{L}_j\rho^2}{P^{n-\mu-1}}\right) =
\frac{1}{\overline{\Phi}^{\mu+1}}\left[
 \frac{-2\delta_{kj}}{P^{n-\mu-1}}+
 \frac{n-\mu-1}{P^{n-\mu}}(L_k\rho^2)(\overline{L}_j\rho^2)\right]
+\frac{1}{\gamma\gamma^{\ast}}\lra_{2},
\end{equation*}
and for $k=n$ and $j<n$, using Lemma \ref{lp}, we have
\begin{equation*}
\frac{1}{\overline{\Phi}^{\mu+1}}\Lambda_n\left(
 \frac{\overline{L}_j\rho^2}{P^{n-\mu-1}}\right)
=\frac{1}{\gamma}\frac{2(n-\mu-1)\overline{L}_j\rho^2}{\overline{\Phi}^{\mu}P^{n-\mu}}
+\frac{1}{\gamma\gamma^{\ast}}\lra_2
.
\end{equation*}
Thus, for $k<n$,
\begin{align}
\label{mkjfrom}
&\lrm_{kj}^{\mu} = \frac{1}{\overline{\Phi}^{\mu+1}}\left[
 \frac{-2\delta_{kj}}{P^{n-\mu-1}}+\frac{n-\mu-1}{P^{n-\mu}}(L_k\rho^2)(\overline{L}_j\rho^2)\right]
+\frac{1}{\gamma\gamma^{\ast}}\lra_{2}
   \\
\nonumber
&\lrm_{nj}^{\mu} =\frac{1}{\gamma}\frac{2(n-\mu-1)\overline{L}_j\rho^2}{\overline{\Phi}^{\mu}P^{n-\mu}}
+\frac{1}{\gamma\gamma^{\ast}}\lra_{2}
,
 \end{align}
 taking into account calculations such as multiplying and the
 dividing by a factor of $\gamma$ in order to obtain a type two
 operator divided by a factor of $\gamma$.

We calculate in a similar manner the $\widetilde{\lrm}_{kj}^{\mu}$ terms.
 For these terms we use
 the symmetry of (\ref{derapp}) to write
\begin{equation*}
\overline{L}_j=\frac{\partial}{\partial \overline{\zeta}_j}
 +(\xi_0\xi_{-1}^{\ast}\sigma_1)\Lambda,
\end{equation*}
and as a consequence
\begin{equation*}
L_k\overline{L}_j\rho^2
 =2\delta_{jk} +\xi_{-1}\xi_0^{\ast}\sigma_1
  +\xi_{-1}\xi_{-1}^{\ast}\sigma_2.
 \end{equation*}

For $k<n$,
\begin{align}
\nonumber
\widetilde{\lrm}_{kj}^{\mu}
 =&\frac{1}{\overline{\Phi}^{\mu+1}}L_k\left(\frac{\overline{L}_j\rho^2}{P^{n-\mu-1}}\right)
 +\frac{1}{\gamma\gamma^{\ast}}\lra_{2}\\
\nonumber
=&\frac{1}{\overline{\Phi}^{\mu+1}}
\frac{L_k\overline{L}_j\rho^2}{P^{n-\mu-1}}-\frac{n-\mu-1}{\overline{\Phi}^{\mu+1}}
 \frac{\overline{L}_j\rho^2}{P^{n-\mu}}L_kP
+\frac{1}{\gamma\gamma^{\ast}}\lra_{2}\\
\label{mkjfrom2}
 =&
 \frac{2\delta_{kj}}{\overline{\Phi}^{\mu+1}P^{n-\mu-1}}
 -(n-\mu-1)\frac{(\overline{L}_j\rho^2)(L_k\rho^2)}{\overline{\Phi}^{\mu+1}P^{n-\mu}}
+\frac{1}{\gamma\gamma^{\ast}}\lra_{2}
.
\end{align}
For $k=n$ and $j<n$, we calculate
\begin{equation*}
 \widetilde{\lrm}_{nj}^{\mu}
 =
\frac{1}{\overline{\Phi}^{\mu+1}}L_n\left(\frac{\overline{L}_j\rho^2}{P^{n-\mu-1}}\right) -
 (\mu+1)\frac{\overline{L}_j\rho^2}{\overline{\Phi}^{\mu+2}P^{n-\mu-1}}L_n\overline{\Phi}.
\end{equation*}
Using Lemma \ref{lphi} we can write
\begin{align}
\label{mntil}
\widetilde{\lrm}_{nj}^{\mu}
 =&\gamma^{\ast}(\mu+1)
 \frac{\overline{L}_j\rho^2}{\overline{\Phi}^{\mu+2}P^{n-\mu-1}}
-
(n-\mu-1)\frac{\overline{L}_j\rho^2}{\overline{\Phi}^{\mu+1}P^{n-\mu}}L_nP
+\frac{1}{\gamma\gamma^{\ast}}\lra_{2}
.
\end{align}
We now use Lemma \ref{lp} to write the second term on the right side of (\ref{mntil})
 as
\begin{equation*}
\frac{2(n-\mu-1)}{\gamma^{\ast}}
\frac{\Phi \overline{L}_j\rho^2}{\overline{\Phi}^{\mu+1}P^{n-\mu}}
 +\frac{1}{\gamma\gamma^{\ast}}\lra_{2}
,
\end{equation*}
and we can then write
\begin{equation*}
\widetilde{\lrm}_{nj}^{\mu}
 =\gamma^{\ast}(\mu+1)
 \frac{\overline{L}_j\rho^2}{\overline{\Phi}^{\mu+2}P^{n-\mu-1}}
 +\frac{2(n-\mu-1)}{\gamma^{\ast}}
\frac{\Phi \overline{L}_j\rho^2}{\overline{\Phi}^{\mu+1}P^{n-\mu}}
+\frac{1}{\gamma\gamma^{\ast}}\lra_{2}
.
\end{equation*}

From Lemma \ref{lemmalq} we have
\begin{align*}
&\vartheta_{\zeta}\lrl_q=-c_{nq}
 \sum_{{k,\mu\atop j\ne n}\atop{|K|=q+1\atop |L|=q}}
 {n-2-\mu \choose q}
\widetilde{\lrm}_{kj}^{\mu}
 \gamma\varepsilon_{njL}^{kK}\overline{\omega}^K\wedge \Theta^L +
\frac{1}{\gamma}\lra_{2}\\
&\partial_z\lrl_{q-1}=c_{n,q-1}
 \sum_{{j,k,\mu}\atop{|K|=q+1\atop |L|=q}}
 {n-2-\mu \choose q-1}
\lrm_{kj}^{\mu}
 \gamma\varepsilon^{njQ}_{K}\varepsilon^{kQ}_{L}
\overline{\omega}^K\wedge \Theta^L + \frac{1}{\gamma}\lra_{2}.
\end{align*}
We separate the terms with $n\in K$ from those with $n\notin K$.
\begin{align}
 \nonumber
&\vartheta_{\zeta}\lrl_q-\partial_z\lrl_{q-1}
 =-\sum_{{ j< n}\atop{n\notin K\atop L}}c_{nq}
 \sum_{\mu}{n-2-\mu \choose q}
\widetilde{\lrm}_{nj}^{\mu}
 \gamma\varepsilon_{jL}^{K}\overline{\omega}^K\wedge \Theta^L\\
\nonumber
 &+\sum_{{j,k< n}\atop{n\notin J\atop n\notin
L}}\Bigg(c_{nq}
 \sum_{\mu}{n-2-\mu \choose q}
\widetilde{\lrm}_{kj}^{\mu}
 \gamma\varepsilon_{jL}^{kJ}\\
 &\qquad \qquad
 \nonumber
-c_{n,q-1}
 \sum_{\mu}
 {n-2-\mu \choose q-1}
\lrm_{kj}^{\mu}
 \gamma\varepsilon^{jQ}_{J}\varepsilon^{kQ}_{L}\Bigg)
\overline{\omega}^{nJ}\wedge \Theta^L\\
 \label{sepn}
&-\sum_{{ j< n}\atop{n\notin J\atop n\notin Q}}c_{n,q-1}
 \sum_{\mu}{n-2-\mu \choose q-1}
\lrm_{nj}^{\mu}
 \gamma\varepsilon_{J}^{jQ}\overline{\omega}^{nJ}\wedge
 \Theta^{nQ}
+\frac{1}{\gamma}\lra_{2}.
\end{align}

We write
\begin{equation*}
  \vartheta_{\zeta}\lrl_q-\partial_z\lrl_{q-1}
 =\sum_{|L|=q}\lrh_L\wedge\Theta^L=\lrh_q
\end{equation*} and compute the $\lrh_L$ terms.

For $n\in L$ we have
\begin{equation*}
\lrh_{nQ}=
 -2^{n-1}\left(\frac{1}{2\pi}\right)^n
 (n-1)!\frac{1}{P^n}\sum_{j<n\atop j\notin Q}
\overline{L}_j\rho^2\overline{\omega}^{njQ}
 +\frac{1}{\gamma^{\ast}}\lra_{2}.
\end{equation*}

For $n\notin L$ we distinguish the following different cases for
the exponent, $K$ of $\overline{\omega}$ in (\ref{sepn}).
\begin{enumerate}
\item[a)] $K=lL$ with $l<n$\\
\item[b)] $K=nL$\\
\item[c)] $K=nJ$, $J\neq L$.
\end{enumerate}
In case $a)$ we have
\begin{align*}
\lrh_{L}=& -\sum_{j<n\atop j\notin L} c_{nq}
 \Bigg(\sum_{\mu} {n-\mu-2 \choose q} \gamma^2(\mu+1)
 \frac{\overline{L}_j\rho^2}{\overline{\Phi}^{\mu+2}P^{n-\mu-1}}\\
 &\qquad
 +2(n-1)\frac{\gamma}{\gamma^{\ast}}
\frac{\Phi \overline{L}_j\rho^2}{\overline{\Phi}P^{n}}
\Bigg)\overline{\omega}^{jL}+\frac{1}{\gamma^{\ast}}\lra_{2}.
\end{align*}

For case $b)$, we have
\begin{align*}
\lrh_{L}=&
 \sum_{j<n\atop n\notin L}\left(c_{nq}
 \sum_{\mu}{n-2-\mu \choose q}
\widetilde{\lrm}_{jj}^{\mu} -c_{n,q-1}
 \sum_{\mu}
 {n-2-\mu \choose q-1}
\lrm_{jj}^{\mu}
 \right)
\gamma\overline{\omega}^{nL}\\
=&2^{n-2}\left(\frac{1}{2\pi}\right)^n (n-2)!\sum_{j<n\atop
n\notin L}
 \left(
 \frac{2}{\overline{\Phi}P^{n-1}}
 -(n-1)
\frac{|L_j\rho^2|^2}{\overline{\Phi}P^{n}}\right)\gamma\overline{\omega}^{nL}
+\frac{1}{\gamma^{\ast}}\lra_{2}
\\
=&2^{n-2}\left(\frac{1}{2\pi}\right)^n (n-1)!
 \frac{1}{\overline{\Phi}P^{n}}\sum_{n\notin
L}  \left(
 2P
 -\sum_{j<n} |L_j\rho^2|^2\right)\gamma\overline{\omega}^{nL}
 +\frac{1}{\gamma^{\ast}}\lra_{2}
\\
=&\frac{1}{\gamma^{\ast}} \frac{2^{n-2}}{(2\pi)^n} (n-1)!
 \frac{4\Phi}{P^n}\overline{\omega}^{nL}
+\frac{1}{\gamma^{\ast}}\lra_{2}\\
=&\frac{1}{\gamma} \frac{2^{n-2}}{(2\pi)^n} (n-1)!
 \frac{4\Phi}{P^n}\overline{\omega}^{nL}
+\frac{1}{\gamma\gamma^{\ast}}\lra_{2} .
\end{align*}

For case $c)$ we write
\begin{align*}
 \lrh_L &=
 \sum_{j\neq k\atop n\notin
L}\left(c_{nq}
 \sum_{\mu}{n-2-\mu \choose q}
\widetilde{\lrm}_{kj}^{\mu}
 \gamma\varepsilon_{jL}^{kJ}
-c_{n,q-1}
 \sum_{\mu}
 {n-2-\mu \choose q-1}
\lrm_{kj}^{\mu}
 \gamma\varepsilon^{jQ}_{J}\varepsilon^{kQ}_{L}\right)
\overline{\omega}^{nJ}\\
&= -
 \sum_{j\neq k\atop n\notin
L}\left(c_{nq}
 \sum_{\mu}{n-2-\mu \choose q}
\widetilde{\lrm}_{kj}^{\mu}
+c_{n,q-1}
 \sum_{\mu}
 {n-2-\mu \choose q-1}
\lrm_{kj}^{\mu}
\right) \gamma\varepsilon_{jL}^{kJ}
\overline{\omega}^{nJ},
\end{align*}
since $\varepsilon_{jL}^{kJ}=-\varepsilon_J^{jQ}\varepsilon_L^{kQ}$ when $k\neq j$.
Thus from (\ref{mkjfrom}) and (\ref{mkjfrom2})
 we can write
\begin{align*}
 \frac{2^{n-2}}{(2\pi)^n} (n-1)!&
 \frac{1}{\overline{\Phi}P^{n}}
 \sum_{j\neq k\atop n\notin L}
 \big(-(\overline{L}_j\rho^2)(L_k\rho^2)
+(\overline{L}_j\rho^2)(L_k\rho^2)\big) \gamma\varepsilon_{jL}^{kJ}
\overline{\omega}^{nJ}
+\frac{1}{\gamma^{\ast}}\lra_{2}
\\
=&\frac{1}{\gamma^{\ast}}\lra_{2}.
\end{align*}

We have therefore established the
\begin{lemma}
\label{hterms} \label{hlem}
$\vartheta_{\zeta}\lrl_q-\partial_z\lrl_{q-1}
 =\sum_{|L|=q}\lrh_L\wedge\Theta^L$ with
\begin{align*}
\lrh_{nQ}=&
 -2^{n-1}\left(\frac{1}{2\pi}\right)^n
 (n-1)!\frac{1}{P^n}\sum_{j<n\atop j\notin Q}
\overline{L}_j\rho^2\overline{\omega}^{njQ}
+\frac{1}{\gamma^{\ast}}\lra_{2}
,
\end{align*}
and for $n\notin L$:
\begin{align*}
 \lrh_L
 =&-\sum_{j<n\atop j\notin L} c_{nq}
 \Bigg(\sum_{\mu} {n-\mu-2 \choose q} \gamma^2
 \frac{(\mu+1)\overline{L}_j\rho^2}{\overline{\Phi}^{\mu+2}P^{n-\mu-1}}\\
 &\qquad
 +2{n-2 \choose q}(n-1)\frac{\gamma}{\gamma^{\ast}}
\frac{\Phi \overline{L}_j\rho^2}{\overline{\Phi}P^{n}}
\Bigg)\overline{\omega}^{jL} +\frac{1}{\gamma}
\frac{2^{n-2}}{(2\pi)^n} (n-1)!
 \frac{4\Phi}{P^n}\overline{\omega}^{nL}
+\frac{1}{\gamma\gamma^{\ast}}\lra_{2}.
\end{align*}
\end{lemma}

\section{The structure of the kernels $\lrt_q$}
In this section we prove Theorem \ref{nkern} which expresses the
kernels, $\lrt_q$ and $\lrt_q^{\ast}$ as derivatives of explicitly
computed simpler kernels.  We solve
\begin{align*}
&\mdbar_{\zeta}\lrn_q = \lrt_q + \frac{1}{\gamma\gamma^{\ast}}
\lra_2\\
&\mdbar_{\zeta}^{\ast}\lrn_q = \lrt_{q-1}^{\ast} +
\frac{1}{\gamma\gamma^{\ast}}
\lra_2\\
 &\lrn_q=\lrn_q^{\ast}+\frac{1}{\gamma}\lra_{3} +
\frac{1}{\gamma^{\ast}}\lra_{3}\\
&\left. \ast\lrn_q\right|_{\partial D}\equiv 0.
\end{align*}

We set
\begin{equation*}
 \lrn_q=\lrg_q+\Gamma_{0q}
\end{equation*}
and determine $\lrg_q$.
With
\begin{equation*}
 \lrg_q=\sum_{|L|=q}\lrg_L\wedge\Theta^L, \qquad \lrg_q=\lrg_q^{\ast}+
\frac{1}{\gamma}\lra_{3}+\frac{1}{\gamma^{\ast}}\lra_{3}
\end{equation*}
we solve
\begin{equation}
 \label{dgh}
 \mdbar_{\zeta}\lrg_L=\lrh_L
 +\frac{1}{\gamma\gamma^{\ast}}\lra_2
\end{equation}
where $\lrh_L$ is as in Lemma \ref{hlem}.


If $q<n-1$ then we obtain (\ref{dgh}) by choosing
\begin{align}
\label{gnq}
 \lrg_{nQ}=&-\frac{2^{n-1}(n-2)!}{(2\pi)^n}
 \frac{1}{P^{n-1}}\overline{\omega}^{nQ}, \qquad L=nQ,\\
\nonumber
 \lrg_{L}=&c_{nq}
 \Bigg(\sum_{\mu} {n-\mu-2 \choose q}
 \gamma^2\frac{\mu+1}{n-\mu-2}
 \frac{1}{\overline{\Phi}^{\mu+2}P^{n-\mu-2}}\\
 \label{gn-1}
 &\qquad
 +{n-2 \choose q}\frac{\gamma}{\gamma^{\ast}}
\frac{2\Phi}{\overline{\Phi}P^{n-1}} \Bigg)\overline{\omega}^{L}
\qquad n\notin L.
\end{align}

We verify (\ref{dgh}) for the case $q<n-1$ by calculating
$\mdbar_x\lrg_L$. That
 $\mdbar_x \lrg_{nQ}=\lrh_{nQ}$ is easy to see, and we turn to
(\ref{dgh}) in the case $n \notin L$.  We have
\begin{align*}
 \mdbar_{\zeta} \lrg_L
 &=
  \sum_{j\notin L} c_{nq}
 \overline{L}_j \Bigg(\sum_{\mu} {n-\mu-2 \choose q}
 \gamma^2\frac{\mu+1}{n-\mu-2}
 \frac{1}{\overline{\Phi}^{\mu+2}P^{n-\mu-2}}\\
 \nonumber
 &\qquad
 +{n-2 \choose q}\frac{\gamma}{\gamma^{\ast}}
\frac{2\Phi}{\overline{\Phi}P^{n-1}}
\Bigg)\overline{\omega}^{jL}\\
\nonumber
 & = \sum_{j\notin L} c_{nq}
 \Bigg(\sum_{\mu} {n-\mu-2 \choose q}
 \gamma^2\frac{\mu+1}{n-\mu-2}
 \overline{L}_j \left(\frac{1}{\overline{\Phi}^{\mu+2}P^{n-\mu-2}}
 \right)\\
 \nonumber
 &\qquad
 +{n-2 \choose q}\frac{\gamma}{\gamma^{\ast}}
\overline{L}_j \left( \frac{2\Phi}{\overline{\Phi}P^{n-1}}\right)
\Bigg)\overline{\omega}^{jL} + \frac{1}{\gamma^{\ast}}\lra_2.
\end{align*}
We consider separately the cases $j<n$ and $j=n$.

In view of Lemma \ref{lphi}, we have, for $j<n$,
\begin{align*}
\gamma^2\overline{L}_j
\left(\frac{1}{\overline{\Phi}^{\mu+2}P^{n-\mu-2}}
 \right)
  =& -(n-\mu-2)\gamma^2 \frac{1}{\overline{\Phi}^{\mu+2}P^{n-\mu-1}}
\overline{L}_jP+\lra_2\\
=&-(n-\mu-2)\gamma^2 \frac{1}{\overline{\Phi}^{\mu+2}P^{n-\mu-1}}
\overline{L}_j\rho^2+\lra_2,
\end{align*}
where the last line follows from Lemma \ref{lp} $ii.$  Similarly,
we have
\begin{align*}
\frac{\gamma}{\gamma^{\ast}}
 \overline{L}_j \left(
\frac{2\Phi}{\overline{\Phi}P^{n-1}}\right)
 &= -(n-1) \frac{\gamma}{\gamma^{\ast}}
  \frac{2\Phi}{\overline{\Phi}P^{n}}\overline{L}_j P +
 \lra_2\\
 &=-(n-1)  \frac{\gamma}{\gamma^{\ast}}
 \frac{2\Phi}{\overline{\Phi}P^{n}}\overline{L}_j \rho^2  +
 \frac{1}{\gamma^{\ast}}\lra_2.
 \end{align*}

We thus far can write
\begin{align}
\label{dbarg}
 \mdbar_{\zeta} \lrg_L =&-\sum_{j<n\atop j\notin L} c_{nq}
 \Bigg(\sum_{\mu} {n-\mu-2 \choose q} \gamma^2
 \frac{(\mu+1)\overline{L}_j\rho^2}{\overline{\Phi}^{\mu+2}P^{n-\mu-1}}\\
 \nonumber
 &\qquad
 +2{n-2 \choose q}(n-1)\frac{\gamma}{\gamma^{\ast}}
\frac{\Phi \overline{L}_j\rho^2}{\overline{\Phi}P^{n}}
\Bigg)\overline{\omega}^{jL}\\
\nonumber
 & + c_{nq}
 \Bigg(\sum_{\mu} {n-\mu-2 \choose q}
 \gamma^2\frac{\mu+1}{n-\mu-2}
 \overline{L}_n \left(\frac{1}{\overline{\Phi}^{\mu+2}P^{n-\mu-2}}
 \right)\\
 \nonumber
 &\qquad
 +{n-2 \choose q}\frac{\gamma}{\gamma^{\ast}}
\overline{L}_n \left( \frac{2\Phi}{\overline{\Phi}P^{n-1}}\right)
\Bigg)\overline{\omega}^{jL} + \frac{1}{\gamma^{\ast}}\lra_2.
\end{align}

 In dealing with $j=n$ we
have
\begin{align}
\label{j=n1}
 \gamma^2 \overline{L}_n
\left(\frac{1}{\overline{\Phi}^{\mu+2}P^{n-\mu-2}}\right)
 &=
 - (n-\mu-2)\gamma^2\frac{1}{\overline{\Phi}^{\mu+2}P^{n-\mu-1}}\overline{L}_n
 P + \lra_2\\
 \nonumber
 &= -2(n-\mu-2) \frac{\gamma^2}{\gamma^{\ast}}
  \frac{1}{\overline{\Phi}^{\mu+1}P^{n-\mu-1}} + \frac{1}{\gamma}
 \lra_2+ \frac{1}{\gamma^{\ast}}
 \lra_2\\
 \nonumber
 &= -2(n-\mu-2) \gamma
  \frac{1}{\overline{\Phi}^{\mu+1}P^{n-\mu-1}} + \frac{1}{\gamma}
 \lra_2 + \frac{1}{\gamma^{\ast}}
 \lra_2,
\end{align}
by Lemma \ref{lp} $i.$  Also, we have
\begin{equation}
 \label{gamst}
\frac{\gamma}{\gamma^{\ast}}
 \overline{L}_n \left(
\frac{2\Phi}{\overline{\Phi}P^{n-1}}\right)
 = -2 \gamma
 \frac{1}{\overline{\Phi}P^{n-1}}
-2(n-1)\frac{\gamma}{\gamma^{\ast}}\frac{\Phi}{\overline{\Phi}P^{n}}
\overline{L}_nP
 +\frac{1}{\gamma^{\ast}} \lra_2
\end{equation}
by Lemma \ref{lphi} $i.$  We now use a variation of Lemma \ref{lp}
$i.$ which follows.  First, using the symmetry involved in
(\ref{lamsym}) we can write
\begin{equation*}
\gamma^{\ast}L_n P=
 -2 \Phi^{\ast}
  +\xi_{-1}\xi_1^{\ast}(P+\lre_2)+\xi_{-1}\xi_1^{\ast}\sigma_2.
\end{equation*}
Now using $\gamma-\gamma^{\ast}=\sigma_1$ we have
\begin{align*}
\gamma \overline{L}_n P=& \gamma^{\ast} \overline{L}_n P
+ \sigma_1 \overline{L}_n P\\
=&-2 \overline{\Phi}
  +\xi_{-1}\xi_1^{\ast}(P+\lre_2)+\xi_{-1}\xi_1^{\ast}\sigma_2
   +\sigma_1 \overline{L}_nP.
\end{align*}
 And so (\ref{gamst}) becomes
\begin{align*}
\frac{\gamma}{\gamma^{\ast}}
 \overline{L}_n \left(
\frac{2\Phi}{\overline{\Phi}P^{n-1}}\right)
 =& -2 \gamma
 \frac{1}{\overline{\Phi}P^{n-1}}
+(n-1)\frac{1}{\gamma^{\ast}}\frac{4\Phi}{P^{n}} +
 \frac{1}{\gamma^{\ast}}\frac{\Phi}{\overline{\Phi}}\sigma_1\overline{L}_nP
 +\frac{1}{\gamma} \lra_2+\frac{1}{\gamma^{\ast}} \lra_2\\
=& -2 \gamma
 \frac{1}{\overline{\Phi}P^{n-1}}
+(n-1)\frac{1}{\gamma}\frac{4\Phi}{P^{n}} +
 \frac{1}{\gamma^{\ast}}\frac{\Phi}{\overline{\Phi}}\sigma_1\overline{L}_nP
 +\frac{1}{\gamma} \lra_2+\frac{1}{\gamma^{\ast}} \lra_2
 .
\end{align*}
We now show the third term on the right can be written as
\begin{equation}
\label{verify}
\frac{\gamma}{\gamma^{\ast}}\frac{\Phi}{\overline{\Phi}}\sigma_1\overline{L}_nP
 =\frac{1}{\gamma\gamma^{\ast}}\lra_2.
\end{equation}
$\Phi$ is a sum of terms of the form
\begin{equation*}
\gamma\xi_0\lre_1+\lre_2-r,
\end{equation*}
and so we consider separately
\begin{align}
\label{1} & \frac{\gamma}{\gamma^{\ast}}
 \frac{\sigma_2}{\overline{\Phi}P^n}\overline{L}_nP\\
 \label{2}
&\frac{1}{\gamma^{\ast}}
 \frac{\sigma_3}{\overline{\Phi}P^n}\overline{L}_nP\\
 \label{3}
&\frac{1}{\gamma^{\ast}}
 \frac{\sigma_1r}{\overline{\Phi}P^n}\overline{L}_nP.
\end{align}
(\ref{2}) leads to the desired error terms with the substitution
\begin{equation*}
\overline{L}_nP = \lre_1+\xi_1^{\ast}.
\end{equation*}
In (\ref{3}) we substitute
\begin{align*}
\overline{L}_n P &= \lre_1 +
 \frac{r^{\ast}}{\gamma^{\ast}} + \frac{1}{\gamma}\xi_0
 (P+\lre_2)\\
 &=\frac{\gamma}{r}(P+\lre_2)+\frac{1}{\gamma}\xi_0
 (P+\lre_2)
\end{align*}
and we obtain
\begin{equation*}
\frac{1}{\gamma^{\ast}}
 \frac{\sigma_1r}{\overline{\Phi}P^n}\overline{L}_nP =
 \frac{1}{\gamma\gamma^{\ast}}\lra_2.
\end{equation*}
Turning now to (\ref{1}), we write
\begin{equation*}
\frac{\gamma}{\gamma^{\ast}}
 = \frac{\gamma^{\ast}}{\gamma}
+\frac{\sigma_1}{\gamma}+\frac{\sigma_1}{\gamma^{\ast}}
\end{equation*}
so as to obtain
\begin{equation*}
\frac{\gamma}{\gamma^{\ast}}
 \frac{\sigma_2}{\overline{\Phi}P^n}\overline{L}_nP
=\frac{\gamma^{\ast}}{\gamma}
 \frac{\sigma_2}{\overline{\Phi}P^n}\overline{L}_nP
 +\frac{1}{\gamma}
 \frac{\sigma_3}{\overline{\Phi}P^n}\overline{L}_nP
 +\frac{1}{\gamma^{\ast}}
 \frac{\sigma_3}{\overline{\Phi}P^n}\overline{L}_nP.
\end{equation*}
We can now substitute
\begin{equation*}
\overline{L}_n P
 = \lre_1 + \xi_0 \frac{r^{\ast}}{\gamma^{\ast}}
\end{equation*}
in the first term on the right hand side, and use
\begin{equation*}
\overline{L}_n P
 = \lre_1 + \xi_0\xi_1^{\ast}
 \end{equation*}
 in the last two terms on the right hand side, in order to
 complete the verification of (\ref{verify}).

 Thus (\ref{gamst}) becomes
 \begin{equation}
 \label{j=n2}
 \frac{\gamma}{\gamma^{\ast}}
 \overline{L}_n \left(
\frac{2\Phi}{\overline{\Phi}P^{n-1}}\right)
 = -2 \gamma
 \frac{1}{\overline{\Phi}P^{n-1}}
+(n-1)\frac{1}{\gamma}\frac{4\Phi}{P^{n}}
 +\frac{1}{\gamma\gamma^{\ast}} \lra_2.
 \end{equation}

 Together, (\ref{j=n1}) and (\ref{j=n2}), when inserted into
 (\ref{dbarg}) give the term
 \begin{equation*}
\frac{1}{\gamma}\frac{2^{n-2}}{(2\pi)^n}(n-1)!
\frac{4\Phi}{P^{n}}\overline{\omega}^{nL}
 +\frac{1}{\gamma\gamma^{\ast}} \lra_2
 \end{equation*}
 which is the remaining part of $\lrh_L$ in (\ref{dgh}).

The calculations leading to the expressions for $\lrg_L$
were done in a special coordinate chart near the boundary.  To
globalize the expressions we note
 there are double forms $\sigma_1$ such that
\begin{equation*}
-\frac{1}{2}\mdbar_{\zeta}\partial_z\rho^2
 = \overline{\omega}^n\wedge\Theta^n + \overline{\omega}^n\wedge\sigma_1
 +\sigma_1\wedge\Theta^n+\tau,
\end{equation*}
where $\tau$ does not contain any $\overline{\omega}^n$ or
$\Theta^n$ terms.  If we set $\nu(\zeta,z)=
\overline{\omega}^n\wedge\Theta^n +
\overline{\omega}^n\wedge\sigma_1
 +\sigma_1\wedge\Theta^n$, we have
\begin{align*}
&\nu(\zeta,z)= \overline{\omega}^n\wedge\Theta^n +\sigma_1\\
&\tau(\zeta,z)= \sum_{j<n} \overline{\omega}^j\wedge\Theta^j
+\sigma_1.
\end{align*}
 We thus have
\begin{align*}
(\mdbar_{\zeta}\partial_z\rho^2)^q&=
 (-2)^q(\tau^q+q\tau^{q-1}\wedge\nu)+\sigma_1\\
\tau^q&=q!\sum_{|L|=q\atop n\notin L} \overline{\omega}^L
 \wedge\Theta^L+\sigma_1\\
\tau^{q-1}\wedge\nu&=(q-1)!\sum_{|Q|=q-1}
 \overline{\omega}^{nQ}
 \wedge\Theta^{nQ}+\sigma_1,
\end{align*}
which we use in connection with (\ref{gnq}) and (\ref{gn-1}) to
write
\begin{prop}
\label{neum}
 Let $n\ge 3$. For $1\le q\le n-2$ let the
differential forms $\lrn_q$ be given by
\begin{align*}
\lrn_q= 2^{n-2}\Bigg(\frac{1}{2\pi}\Bigg)^n&(n-q-2)!
 \Bigg(\sum_{0\le\mu\le n-q-2}\gamma^2 {n-\mu -2 \choose q}
 \frac{\mu+1}{n-\mu-2}\frac{1}{\overline{\Phi}^{\mu+2}P^{n-\mu-2}}\\
&+{n-2 \choose q}\frac{\gamma}{\gamma^{\ast}}
\frac{2\Phi}{\overline{\Phi}P^{n-1}} \Bigg)\tau^q
 -\frac{2^{n-1}(n-2)!}{(q-1)!(2\pi)^n}\frac{1}{P^{n-1}}\tau^{q-1}\wedge\nu
  +\Gamma_{0q}.
\end{align*}
Then the $\lrn_q$ fulfills the first set of equations of Theorem
\ref{nkern}.
\end{prop}
\begin{remark} While $\lrn_0$ can also be explicitly computed,
such a term will involve logarithms and does not fit the
definition of admissible operators.  Nonetheless, similar mapping
properties for such operators do exist.  $\lrn_{n-1}$ can also be
explicitly given and can be handled as in \cite{LiMi} combined
with the above methods.  The estimates remain true but the
principal part of $\lrn_{n-1}$ changes.
\end{remark}

We now verify - in order to complete the proof of Theorem
\ref{nkern} -
\begin{equation*}
\mdbar_{\zeta}^{\ast}\lrn_q=\lrt_{q-1}^{\ast}
 +\frac{1}{\gamma\gamma^{\ast}}\lra_{2}
\end{equation*}
for $1\le q\le n-2$, by showing
\begin{equation}
\label{hq=dg}
 \vartheta_{\zeta} \lrg_q=\lrh_{q-1}^{\ast}+
 \frac{1}{\gamma\gamma^{\ast}}\lra_{2}.
\end{equation}
From Lemma \ref{hterms} we have
\begin{align}
\nonumber
 \lrh_{q-1}^{\ast}=&
 \frac{2^{n-1}}{(2\pi)^n}(n-1)!
 \Bigg(\sum_{n\in L\atop{j<n\atop |L|=q-1}}
 \frac{\Lambda_j\rho^2}{P^n} \overline{\omega}^l\wedge\Theta^{jL}
 +\sum_{n\notin L\atop |L|=q-1}\frac{1}{\gamma^{\ast}}
\frac{2\Phi}{P^n} \overline{\omega}^L\wedge \Theta^{nL}\\
\nonumber
&
 -\sum_{n\notin L\atop{j<n\atop |L|=q-1}}
\frac{\gamma^{\ast}}{\gamma} \frac{\Phi\Lambda_j\rho^2}
{\overline{\Phi}P^n}\overline{\omega}^{L}\wedge\Theta^{jL}\Bigg)\\
\nonumber
&-\sum_{n\notin L\atop{j<n \atop|L|=q-1}} c_{n,q-1}\sum_{\mu}
 {n-\mu-2\choose q-1}
(\gamma^{\ast})^2\frac{(\mu+1)\Lambda_j\rho^2}{\overline{\Phi}^{\mu+2}P^{n-\mu-1}}
\overline{\omega}^L\wedge\Theta^{jL}\\
\label{hq} & +\frac{1}{\gamma\gamma^{\ast}}\lra_{2}
\end{align}
for $q\ge 2$.
 The case $q=1$ has a similar expression.

 From (\ref{gnq}) and (\ref{gn-1}) we calculate
\begin{align}
\nonumber
 \vartheta_{\zeta}\lrg_q=&
 \frac{2^{n-1}(n-2)!}{(2\pi)^n} \sum_{n\in J\atop j,L}
 L_j\left(\frac{1}{P^{n-1}}\right)\varepsilon_J^{jL}
 \overline{\omega}^L\wedge\Theta^J\\
\nonumber
&-\sum_{n\notin J\atop j,L}c_{nq}\Bigg(
 \sum_{\mu}{n-\mu-2\choose q}\gamma^2 \frac{\mu+1}{n-\mu-2}
 L_j \left(\frac{1}{\overline{\Phi}^{\mu+2}P^{n-\mu-2}}\right)\\
\nonumber &+\frac{\gamma}{\gamma^{\ast}} {n-2\choose q}
L_j\left(\frac{2\Phi}{\overline{\Phi}P^{n-1}}\right)
 \Bigg) \varepsilon^{jL}_J\overline{\omega}^L\wedge\Theta^J
 +\frac{1}{\gamma^{\ast}}\lra_{2}\\
\nonumber =&\frac{2^{n-1}(n-1)!}{(2\pi)^n}\left( \sum_{n\in J\atop
L} \left(\sum_{j<n}
 \frac{\Lambda_j\rho^2}{P^{n}}
 \overline{\omega}^L\wedge\varepsilon_J^{jL}\Theta^J
\right)+\frac{1}{\gamma^{\ast}}\frac{2\Phi}{P^n}
\overline{\omega}^L\wedge\varepsilon_J^{nL}\Theta^J \right)\\
\nonumber
&-\sum_{n\notin J\atop j,L}c_{nq}\Bigg[
 \sum_{\mu}{n-\mu-2\choose q}\gamma^2
\left( \frac{(\mu+1)(\mu+2)}{n-\mu-2}
 \frac{\Lambda_j\rho^2}{\overline{\Phi}^{\mu+3}P^{n-\mu-2}}
 +\frac{(\mu+1)\Lambda_j\rho^2}{\overline{\Phi}^{\mu+2}P^{n-\mu-1}}\right)\\
&+\frac{\gamma}{\gamma^{\ast}} {n-2\choose q} (n-1)
\frac{2\Phi\Lambda_j\rho^2}{\overline{\Phi}P^{n}}
  \overline{\omega}^L\wedge\varepsilon^{jL}_J\Theta^J
\Bigg]+\frac{1}{\gamma}\lra_{2}+\frac{1}{\gamma^{\ast}}\lra_{2}.
\label{dg}
\end{align}
We use
\begin{equation*}
 \frac{\gamma}{\gamma^{\ast}}
 \frac{\Phi\Lambda_j\rho^2}{\overline{\Phi}P^{n}}
=\frac{\gamma^{\ast}}{\gamma}
 \frac{\Phi\Lambda_j\rho^2}{\overline{\Phi}P^{n}}
+ \frac{1}{\gamma}\lra_{2}+\frac{1}{\gamma^{\ast}}\lra_{2}
\end{equation*}
to compare (\ref{hq}) and (\ref{dg}), and to show (\ref{hq=dg})
holds if
\begin{align*}
 \sum_{n\notin L\atop{j<n \atop|L|=q-1}} c_{n,q-1}&\sum_{0\le\mu\le n-q-1}
 {n-\mu-2\choose q-1}
(\gamma^{\ast})^2\frac{(\mu+1)\Lambda_j\rho^2}{\overline{\Phi}^{\mu+2}P^{n-\mu-1}}
\overline{\omega}^L\wedge\Theta^{jL}=\\
&\sum_{n\notin J\atop j,L}c_{nq}
 \sum_{0\le\mu\le n-q-2}{n-\mu-2\choose q}\gamma^2
\Bigg[ \frac{(\mu+1)(\mu+2)}{n-\mu-2}
 \frac{\Lambda_j\rho^2}{\overline{\Phi}^{\mu+3}P^{n-\mu-2}}\\
&\qquad
+\frac{(\mu+1)\Lambda_j\rho^2}{\overline{\Phi}^{\mu+2}P^{n-\mu-1}}\Bigg]
 \overline{\omega}^L\wedge\Theta^{jL}+
 \frac{1}{\gamma}\lra_{2},
\end{align*}
which is an elementary computation.
 The proof of Theorem \ref{nkern} is complete.

\section{$Z$-operators and principal parts}
\label{zop}
 We generalize (slightly) the notion of an isotropic kernel (resp.
 operator).
 We let $\lre_{j-2n}^i(\zeta,z)$ be a
kernel of the form
\begin{equation*}
\lre_{j-2n}^i(\zeta,z)=
 \frac{\sigma_{m}(\zeta,z)}{\rho^{2k}(\zeta,z)} \qquad j\ge 1,
\end{equation*}
where $m-2k\ge j-2n$. We denote by $E_{j-2n}$ the corresponding
 operator.  The following theorem follows from
\cite{LiMi} (see Theorem VII.4.1).

\begin{thrm} \label{E1properties}
 The integral operators $E_{1-2n}$ are continuous from
\begin{equation*}
E_{1-2n}:L^p(D)\rightarrow L^s(D)
\end{equation*}
for any $1\le p\le s\le\infty$ with $1/s>1/p-1/2n$.
\end{thrm}

 We denote by $Z_1$ those operators which are of the form
\begin{equation*}
Z_1=A_{1}+E_{1-2n}.
\end{equation*}
  $Z_2$ operators are defined
similarly,
 and we define $Z_j$, $j>2$, operators by induction to be
those operators of the form
\begin{equation*}
 Z_j=Z_{i_1}\circ\cdots\circ Z_{i_k}\qquad i_1+\cdots+i_k=j.
\end{equation*}

We have the following mapping properties for $Z_j$ operators:
\begin{prop}
\label{zmap}
 Let $p\ge 2$.
\begin{equation*}
Z_j:L^{p}(D)\rightarrow L^{q}(D)
\end{equation*}
where
\begin{equation*}
\frac{1}{q}>\frac{1}{p}-\frac{j}{2n+2}.
\end{equation*}
\end{prop}

We will also use the following property which commutes factors of
$\gamma$ with $Z$ operators.
\begin{lemma}
\label{gamcom}
 Let $m\ge 0$, $k\ge 1$.
\begin{equation*}
\gamma^m Z_k = Z_k\circ \gamma^m + Z_{k+1}.
\end{equation*}
\end{lemma}
\begin{proof}
 The proof follows from the relation
 \begin{equation*}
 \gamma^m(z)=\gamma^m(\zeta)+\sigma_1.
\end{equation*}
\end{proof}

Let $A$ be one of the operators $N_q$, $\mdbar N_q$ and
$\mdbar^{\ast} N_q$ which arise in the \dbar-Neumann problem.  We
are going to describe $A$ in terms of $Z$-operators; this will
show that its continuity properties in weighted $L^p$ spaces
coincide with the behavior of $Z$-operators (although $A$ itself
is not a $Z$-operator).  To proceed we need some more definitions.
\begin{defn} Let $m\le k$ be nonnegative integers.  An operator
\begin{equation*}
C=Z_m +\sum_{\alpha} Z_k^{\alpha}\circ K_{\alpha},
\end{equation*}
where the $Z_j$ are $Z$-operators of type $\ge j$, and the
$K_{\alpha}$ are $L^2$-bounded, is called a $k$-asymptotic
$Z$-operator of type $\ge m$.  We shall denote a generic
$k$-asymptotic $Z$-operator of type $m$ by $C_m^{(k)}$.
\end{defn}

\begin{defn}
Let $A^0$ be an $L^2$-bounded operator.  $A^0$ is a generalized
$Z$-operator if there is an integer $l\ge 0$ such that $\gamma^l
A^0$ is a $Z$-operator.  If $l$ is chosen minimal with that
property, then the type of $\gamma^l A^0$ is called the type of
$A^0$.
\end{defn}
Note that a $Z_m$-operator is $k$-asymptotic $Z$ for any $k$ and
is also a generalized $Z$-operator of type $\ge m$: the integer
$l$ must be 0.

\begin{defn} Let $A$ be an $L^2$-bounded operator.  An
$L^2$-bounded operator $A^0$ is called a {\it principal part} of
$A$, if
\begin{enumerate}
\item[i)] $A^0$ is a generalized $Z$-operator of type $m$,
 \item[ii)] for each $l$, there is an $L$ and an $l$-asymptotic
 $Z$-operator of type $\ge m+1$, $C_{m+1}^{(l)}$, such that
\begin{equation*}
\gamma^L A= \gamma^L A^0 +C^{(l)}_{m+1}.
\end{equation*}
\end{enumerate}
\end{defn}

 It now follows
\begin{thrm}
Let $A$ be an operator with principal part $A^0$ of type $m=1$ or
2.  Let $p\ge 2$ be given.  Then there is an $L$ such that
\begin{equation*}
\gamma^L A: L^p(D)\rightarrow L^q(D),
\end{equation*}
continuously, with
\begin{equation*}
\frac{1}{q}> \frac{1}{p} -\frac{m}{2n+2}.
\end{equation*}
\end{thrm}
This follows from the well-established properties of $Z$-operators
of positive type.

In general, $A$ does not have a principal part.  If $A$ itself is
a generalized $Z$-operator, it is, naturally, its own principal
part.  We do not claim that principal parts are unique - in fact,
they are not.  However, we do have the following theorem regarding
principal operators which tells that the type of a principal part
of $A$ is a property of $A$.
\begin{thrm}
Let $A$ be an operator with principal part $A^0$ of type $m$.
Then $A^0$ is
 unique modulo generalized
$Z$-operators of type $m+1$.
\end{thrm}
\begin{proof}
 By hypothesis, the type of $A^0$ is $m$.  Suppose further that $A$ can
 also be written with a principal part $B^0$ of type $m'\ge m$.
By definition, for each $k,l$, we can find an $K,L$ such that
\begin{align*}
&\gamma^K A= \gamma^K A^0 +C^{(k)}_{m+1}\\
&\gamma^L A= \gamma^L B^0 +C^{(l)}_{m'+1}.
\end{align*}
Let $M=\max (K,L)$.  Then
\begin{equation}
\label{start}
 \gamma^M B^0 +C^{(l)}_{m'+1}=\gamma^M A^0
+C^{(k)}_{m+1}.
\end{equation}
The terms
 $C^{(l)}_{m'+1}$ and $C^{(k)}_{m+1}$ above themselves may be
 written as
\begin{align}
 \label{co1}
&C^{(l)}_{m'+1}=Z_{m'+1} + \sum_{\alpha} Z_l^{\alpha}\circ
K_{\alpha},\\
\label{co2}
 &C^{(k)}_{m+1}=Z_{m+1} + \sum_{\alpha}
Z_k^{\alpha}\circ K_{\alpha}.
\end{align}

Insert (\ref{co1}) and (\ref{co2}) into
 (\ref{start}) to get
\begin{equation*}
 \gamma^M B^0 +Z_{m'+1} + \sum_{\alpha} Z_l^{\alpha}\circ
K_{\alpha}=\gamma^M A^0 +Z_{m+1} + \sum_{\alpha} Z_k^{\alpha}\circ
K_{\alpha}
\end{equation*}
then rearrange to get
\begin{equation}
 \label{rearrange}
 \gamma^M B^0 +Z_{m'+1} - \gamma^M A^0 -Z_{m+1}
 =  \sum_{\alpha} Z_k^{\alpha}\circ K_{\alpha} -  \sum_{\alpha} Z_l^{\alpha}\circ
K_{\alpha}
\end{equation}
The left hand side is an operator of type $m$, and the right hand
side is a $j=\min(k,l)$ - asymptotic operator.

Therefore (\ref{rearrange}) is of the form
\begin{equation}
 \label{zm}
\gamma^M A^j_m = C_{j} ^{(j)},
\end{equation}
where $j=\min(k,l)$, and $A^j_m$ is a $Z$-operator of type $m$. We
note that the operator $A^j_m$ may change with different $j$.

The idea is to show that
 $\gamma^M B^0  - \gamma^M A^0 $, and therefore $\gamma^M A^j_m$,
  is an operator of type $m+1$.  If
 we suppose that it is not, we arrive at a contradiction
 by showing some property of $C_{j} ^{(j)}$ is not exhibited by
$\gamma^M A^j_m$.

Now choose $j$ sufficiently large, and an appropriately large
 $M$ in (\ref{zm}), such that for $s\le m+2$ and all differential
 operators, $D^s$ of order $s$, we have
\begin{equation*}
D^s C_{j} ^{(j)}:L^2(D) \rightarrow L^{\infty}(D),
\end{equation*}
as can be seen by differentiating under the integral.

 Under the assumption that $\gamma^M A^j_m$
is not of type $m+1$,
 we can show there is a differential
operator, $D^s$, of order $s\le m+2$, such that
\begin{equation}
 \label{ds}
D^s A^j_m:  L^2(D) \nrightarrow L^{\infty}(D),
\end{equation}
contradicting (\ref{zm}).

  Since we have the option
of multiplying (\ref{zm}) by factors of $\gamma$, we will ignore
all factors of $\gamma$ which arise in the kernels or by
differentiating such kernels.

 We first note that, modulo factors of $\gamma$,
\begin{equation*}
|\lra_m^j|\lesssim \frac{1}{|\zeta-z|^t}
\end{equation*}
for some integral $t$.
 From (\ref{zm}) we must have, modulo factors of $\gamma$,
 \begin{equation*}
|\lra_m^j| \lesssim \frac{1}{|\zeta-z|^{n}}
\end{equation*}
since, otherwise $A_m^j$ applied to the function
 $1/|\zeta-\zeta_0|^{n-1}$,
 for $\zeta_0\in D$, does not land in $L^{\infty}(D)$, whereas the
 right hand side of (\ref{zm}) applied to the same function is in
 $L^{\infty}(D)$.

 Also, from our assumption that $\gamma^M A^j_m$ is not of type
 $m+1$, and from
  the examination of operators of type
 $m+1$, we have
 \begin{equation}
 \label{lower}
\frac{1}{|\zeta-z|^{n+1-(m+1)/2}} \lesssim |\lra_m^j|.
\end{equation}

If there is no such $D^s$, for $s\le m+2$, for which (\ref{ds})
holds then, since $D^s \lra_m^j$ is bounded by an integer power of
$|\zeta-z|$, we must have
\begin{equation}
\label{alld} |D^s \lra_m^j| \lesssim \frac{1}{|\zeta-z|^{n}}
\end{equation}
for all differential operators of order $s\le m+2$.
 But then (\ref{alld}) implies
\begin{equation*}
|\lra_m^j| \lesssim \frac{1}{|\zeta-z|^{n-(m+2)}},
\end{equation*}
which contradicts (\ref{lower}).
\end{proof}

\section{Integral representations}
\label{repprin}

From \cite{Eh10} we have the explicit version of Theorem
\ref{lrethrm}:
\begin{thrm}  Let $f\in L^2_{0,q}(D)\cap
\mbox{Dom}(\mdbar^{\ast})\cap\mbox{Dom}(\mdbar)$ .  For $1\le q\le
n-2$,
 \label{bir}
\begin{align*}
f_{}(z)=&(\mdbar f_{},\lrt^{}_q)
+(\mdbar_{}^{\ast}f_{},(\lrt^{}_{q-1})^{\ast}) \nonumber
 +
 \left(\mdbar f_{},\frac{1}{\gamma^{\ast}}\lra_{2}^{} + \lre_{2-2n}\right)
 +(\mdbar_{}^{\ast} f_{},\lre_{2-2n})\\
 &+\left(f_{},
\frac{1}{\gamma^{\ast}}\lra_{1}^{}+
\frac{1}{\gamma\gamma^{\ast}}\lra_{2}^{}+
 \lre_{1-2n}\right)
 .
\end{align*}
\end{thrm}

From now on we work with the case $q<n-2$.  The case $q=n-2$ is
somewhat exceptional and can be handled as in \cite{LR87}.  We do
not pursue that case here.

From Theorem \ref{nkern}
 we can write this in terms of $\lrn_q^{}$ as
\begin{align}
\nonumber
 f_{}(z)=&(\mdbar
f_{},\mdbar\lrn^{}_q) +(\mdbar_{}^{\ast}f_{},
\mdbar^{\ast}_{}\lrn^{}_{q})
\\
\nonumber
 &
 +
 \left(\mdbar f_{},\frac{1}{\gamma\gamma^{\ast}}\lra_{2}^{}
 + \lre_{2-2n}\right)
 +\left(\mdbar_{}^{\ast} f_{},
\frac{1}{\gamma\gamma^{\ast}}\lra_{2}^{}
 + \lre_{2-2n}\right)\\
 \label{starthere}
 &
 +\left(f_{},
\frac{1}{\gamma^{\ast}}\lra_{1}^{}+
\frac{1}{\gamma\gamma^{\ast}}\lra_{2}^{}+
 \lre_{1-2n}\right)
,
\end{align}
where the $\lra_{2}^{}$ kernels are such that
\begin{equation*}
 \mdbar\lra_{2}^{}=\lra_{1}^{}+\frac{1}{\gamma}\lra_{2}^{}.
\end{equation*}

\begin{lemma}
\label{gamma3}
\begin{align*}
& i)\ \gamma^3 f = \gamma^{\ast} (\gamma^2 \square f, \lrn_q) +
Z_2 \mdbar f
 +Z_2 \mdbar^{\ast} f + Z_1 f\\
& ii)\ \gamma^3 \mdbar f = Z_1 \gamma^2 \square f + Z_1  \mdbar f
 + Z_1 \mdbar^{\ast} f \\
& iii)\ \gamma^3 \mdbar^{\ast} f = Z_1 \gamma^2 \square f + Z_1
\mdbar f
 + Z_1 \mdbar^{\ast} f.
\end{align*}
\end{lemma}
\begin{proof}
 $i)$.
Start with the equation \ref{starthere}:
\begin{align*}
\nonumber
 f_{}(z)=&(\mdbar
f_{},\mdbar\lrn^{}_q) +(\mdbar_{}^{\ast}f_{},
\mdbar^{\ast}_{}\lrn^{}_{q})
\\
\nonumber
 &
 +
 \left(\mdbar f_{},\frac{1}{\gamma\gamma^{\ast}}\lra_{2}^{}
 + \lre_{2-2n}\right)
 +\left(\mdbar_{}^{\ast} f_{},
\frac{1}{\gamma\gamma^{\ast}}\lra_{2}^{}
 + \lre_{2-2n}\right)\\
 \nonumber
 &
 +\left(f_{},
\frac{1}{\gamma^{\ast}}\lra_{1}^{}+
\frac{1}{\gamma\gamma^{\ast}}\lra_{2}^{}+
 \lre_{1-2n}\right)
.
\end{align*}
Apply to $\gamma^2 f$, using
\begin{align*}
&\mdbar \lrn_q = \frac{\gamma}{\gamma^{\ast}} \lra_{1}
 +\frac{1}{\gamma}\lra_{2}
 +\frac{1}{\gamma^{\ast}}\lra_{2} + \lre_{1-2n}\\
&\mdbar^{\ast} \lrn_q = \frac{\gamma}{\gamma^{\ast}} \lra_{1}
 +\frac{1}{\gamma}\lra_{2}
 +\frac{1}{\gamma^{\ast}}\lra_{2} + \lre_{1-2n}:
\end{align*}

\begin{align}
\nonumber
 \gamma^2 f_{}(z)=&(\gamma^2 \mdbar
f_{},\mdbar\lrn^{}_q) +(\gamma^2 \mdbar_{}^{\ast}f_{},
\mdbar^{\ast}_{}\lrn^{}_{q})
\\
\nonumber
 &
 +
 \left(\mdbar f_{},\frac{\gamma}{\gamma^{\ast}}\lra_{2}^{}
 + \lre_{2-2n}\right)
 +\left(\mdbar_{}^{\ast} f_{},
\frac{\gamma}{\gamma^{\ast}}\lra_{2}^{}
 + \lre_{2-2n}\right)\\
 \label{tointpart}
 &
 +\left( f_{},
\frac{\gamma^2}{\gamma^{\ast}}\lra_{1}^{}+
\frac{1}{\gamma^{\ast}}\lra_{2}^{}+
 \lre_{1-2n}\right)
.
\end{align}
Multiplying (\ref{tointpart}) by $\gamma^{\ast}$ gives us
\begin{align*}
\gamma^3 f_{}(z)=&\gamma^{\ast}(\gamma^2 \mdbar
f_{},\mdbar\lrn^{}_q) +\gamma^{\ast}(\gamma^2
\mdbar_{}^{\ast}f_{}, \mdbar^{\ast}_{}\lrn^{}_{q})
 +Z_2 \mdbar f + Z_2  \mdbar^{\ast} f
+Z_1  f \\
=& \gamma^{\ast} (\gamma^2 \square f, \lrn_q) + Z_2 \mdbar f
 +Z_2 \mdbar^{\ast} f + Z_1 f.
\end{align*}

To establish $ii)$, we integrate by parts in the fourth term on
the right of (\ref{tointpart})
 and use the fact that the $\lra_2$ term satisfies
\begin{equation*}
\mdbar \lra_2 =\lra_1+ \frac{1}{\gamma}\lra_2.
\end{equation*}
We obtain
\begin{align*}
\nonumber
 \gamma^2 f_{}(z)=&(\gamma^2 \mdbar
f_{},\mdbar\lrn^{}_q) +(\gamma^2 \mdbar_{}^{\ast}f_{},
\mdbar^{\ast}_{}\lrn^{}_{q})
\\
\nonumber
 &
 +
 \left(\mdbar f_{},\frac{\gamma}{\gamma^{\ast}}\lra_{2}^{}
 + \lre_{2-2n}\right)
 +\left(f_{},
\frac{\gamma}{\gamma^{\ast}}\lra_{1}^{} +
 \frac{1}{\gamma^{\ast}}\lra_{2}^{}
 + \lre_{1-2n} \right)
.
\end{align*}
Then after multiplying by $\gamma^{\ast}$ we have
\begin{equation}
\label{todbar}
 \gamma^3 f_{}(z)=\gamma^{\ast}(\gamma^2 \mdbar
f_{},\mdbar\lrn^{}_q) +\gamma^{\ast}(\gamma^2
\mdbar_{}^{\ast}f_{}, \mdbar^{\ast}_{}\lrn^{}_{q})
 +Z_2 \mdbar f
+Z_1  f .
\end{equation}
We now apply (\ref{todbar}) to $\mdbar f$:
\begin{align}
\nonumber
 \gamma^3 \mdbar f(z)=& \gamma^{\ast}
 (\gamma^2 \mdbar^{\ast}\mdbar f,
\mdbar^{\ast}\lrn_{q+1}) +Z_1 \mdbar f\\
\label{z1gam}
 =&\gamma^{\ast}
 (\gamma^2 \square f,
\mdbar^{\ast}\lrn_{q+1})+ Z_1  \mdbar f
 + Z_1 \mdbar^{\ast} f.
\end{align}
Now use
\begin{equation*}
\mdbar^{\ast} \lrn_{q+1} = \frac{\gamma}{\gamma^{\ast}} \lra_{1}
 +\frac{1}{\gamma}\lra_{2}
 +\frac{1}{\gamma^{\ast}}\lra_{2} + \lre_{1-2n},
\end{equation*}
which follows from Proposition \ref{neum}.

The term $\gamma^{\ast} (\gamma^2\square f, \mdbar^{\ast}
 \lrn_{q+1})$ is then written
 \begin{equation*}
\gamma^{\ast} (\gamma^2\square f, \mdbar^{\ast}
 \lrn_{q+1}) = Z_1 \gamma^2 \square f +
  (\gamma \square f,\lra_2).
 \end{equation*}
The last term can be written as
\begin{align*}
 (\gamma \square f,\lra_2)
  =&  (\mdbar\mdbar^{\ast} f,\gamma \lra_2) +
   (\mdbar^{\ast}\mdbar f, \gamma \lra_2)\\
  =&  (\mdbar^{\ast}f, \gamma \lra_1+\lra_2)
   + (\mdbar f, \gamma \lra_1+\lra_2)\\
   =& Z_1 \mdbar f + Z_1 \mdbar^{\ast} f.
 \end{align*}

 Putting everything together we write
 \begin{equation*}
\gamma^{\ast} (\gamma^2\square f, \mdbar^{\ast}
 \lrn_{q+1}) = Z_1 \gamma^2 \square f +  Z_1 \mdbar f + Z_1 \mdbar^{\ast}
 f,
\end{equation*}
and so by (\ref{z1gam}) we have
\begin{equation*}
\gamma^3 \mdbar f(z)=Z_1 \gamma^2 \square f +  Z_1 \mdbar f + Z_1
\mdbar^{\ast}
 f.
\end{equation*}

In the same way, we obtain $iii)$.
\end{proof}
\begin{thrm}
 \label{mainint}
\begin{align*}
 \nonumber
&i)\ \gamma^{3j}f
 =\gamma^{\ast}(\gamma^{3j-1}\square f, \lrn_q)
 +\sum_{k=3}^{j+1} Z_k\gamma^{3(j-k)+5}\square f\\
 &\qquad\qquad
 +Z_{j+1} \mdbar f + Z_{j+1}\mdbar^{\ast}f +
 Z_j f\\
 \nonumber
&ii)\ \gamma^{3j}\mdbar f
 =\sum_{k=1}^j Z_k \gamma^{3(j-k)+2}\square f + Z_j \mdbar f + Z_j
 \mdbar^{\ast} f\\
 \nonumber
&iii)\ \gamma^{3j}\mdbar^{\ast} f
 =\sum_{k=1}^j Z_k\gamma^{3(j-k)+2}\square f + Z_j \mdbar f + Z_j
 \mdbar^{\ast} f
 \end{align*}
\end{thrm}
\begin{proof}
$i)$. The proof is by induction, the first step being Lemma
\ref{gamma3} $i)$.

To prove $i)$ for $j+1$ we multiply by $\gamma^3$ and commute the
$\gamma$'s with the $Z_k$ operators using repeatedly $\gamma^3 Z_k
= Z_k \gamma^3 + Z_{k+1}$.

Since $\gamma^{\ast} \lrn_q=Z_2$, in light of the above, we can
write
\begin{equation*}
(\gamma^{\ast})^3
 \gamma^{\ast}(\gamma^{3j-1}\square f, \lrn_q)=
\gamma^{\ast}(\gamma^{3(j+1)-1}\square f,
\lrn_q)+Z_3\gamma^{3j-1}\square f.
\end{equation*}

 Thus multiplying $i)$ by $\gamma^3$ gives
\begin{align}
\nonumber
 \gamma^{3(j+1)}f
 =&\gamma^{\ast}(\gamma^{3(j+1)-1}\square f, \lrn_q)+Z_3\gamma^{3j-1}\square f
 \\
\nonumber
 & +\sum_{k=3}^{j+1} Z_k\gamma^{3(j+1-k)+5}\square f
 +\sum_{k=3}^{j+1} Z_{k+1}\gamma^{3(j+1-k)+5}\square f\\
\nonumber
 & +Z_{j+1} \gamma^3 \mdbar f + Z_{j+1}\gamma^3
\mdbar^{\ast}f +
 Z_j \gamma^3 f\\
\nonumber
 & +Z_{j+2}  \mdbar f + Z_{j+2} \mdbar^{\ast}f +
 Z_{j+1}  f\\
\nonumber
 =&\gamma^{\ast}(\gamma^{3(j+1)-1}\square f, \lrn_q) +
 \sum_{k=3}^{j+2} Z_{k}\gamma^{3(j+1-k)+5}\square f\\
\nonumber
 & +Z_{j+1} \gamma^3 \mdbar f + Z_{j+1}\gamma^3
\mdbar^{\ast}f +
 Z_j \gamma^3 f\\
\label{insert}
 & +Z_{j+2}  \mdbar f + Z_{j+2} \mdbar^{\ast}f +
 Z_{j+1}  f
 \end{align}
Inserting the expressions for $\gamma^3f$, $\gamma^3 \mdbar f$ and
$\gamma^3 \mdbar^{\ast} f$ from Lemma \ref{gamma3}, we can write
\begin{align*}
&Z_j \gamma^3 f=
 Z_{j+2} \gamma^2 \square f +
 Z_{j+2}\mdbar f +Z_{j+2}\mdbar^{\ast} f
 +Z_{j+1} f\\
&Z_{j+1} \gamma^3 \mdbar f=
 Z_{j+2} \gamma^2 \square f +
 Z_{j+2}\mdbar f +Z_{j+2}\mdbar^{\ast} f\\
&Z_{j+1} \gamma^3 \mdbar^{\ast} f=
 Z_{j+2} \gamma^2 \square f +
 Z_{j+2}\mdbar f +Z_{j+2}\mdbar^{\ast} f.
\end{align*}
Finally, inserting these into (\ref{insert}) we obtain $i)$.

$ii)$ and $iii)$ are proved similarly.
\end{proof}

\section{Asymptotic development of the Neumann operator}
Let
\begin{equation*}
 H_q: L^2_{0,q}\rightarrow \mathbb{H}^q \qquad q\ge 1
\end{equation*}
denote the orthogonal projection operator onto the harmonic space.
We first note that in $\mathbb{C}^n$, the projection $H_q$ is just
 the 0 operator, however, we include this consideration of the
projection in order that the proof of Theorem \ref{intmain} below
goes through also in the case of domains in complex manifolds.

 Since for $f\in\mathbb{H}^q$, $\mdbar f$ and $\mdbar^{\ast} f$
 vanish,
 from Theorem \ref{bir} we conclude
\begin{equation*}
 \gamma^2 H_q f = Z_1 H_q f.
\end{equation*}
Multiplying by $\gamma^2$ and commuting the $\gamma^2$ with the
$Z_1$-operator leads to
\begin{align*}
\gamma^4 H_q f &= \gamma^2Z_1 H_q f\\
\gamma^4 H_q f &= Z_1 \gamma^2 H_q f +Z_2 H_q f\\
\gamma^4 H_q f &= Z_2 H_q f.
\end{align*}
An induction argument then yields
\begin{equation*}
\gamma^{2j}H_q f=Z_j H_qf.
\end{equation*}

\begin{thrm}
\label{intmain}
\begin{align*}
&i)\ \gamma^{3j} N_q f = \gamma^{\ast} (\gamma^{3j-1} f, \lrn_q)
 +Z_3\gamma^2 f+ C_{j}^{(j)} f \\
&ii)\ \gamma^{3j}\mdbar^{\ast} N_q f
 = \gamma^{\ast}(\gamma^{3j-1} f, \lrt_{q-1}) +Z_2\gamma^2 f +C_{j}^{(j)} f \\
&iii)\ \gamma^{3j}\mdbar N_q f
 = \gamma^{\ast}(\gamma^{3j-1} f, \lrt_{q}^{\ast}) +Z_2\gamma^2 f+
 +C_{j}^{(j)} f .
\end{align*}
\end{thrm}
\begin{proof}
$i.$
 We apply Theorem \ref{mainint} to $N_q f$:
\begin{align*}
 \gamma^{3j} N_q f =&
 \gamma^{\ast} (\gamma^{3j-1}\square N_q f, \lrn_q)
 +\sum_{k=3}^{j+1} Z_k\gamma^{3(j-k)+5}\square N_q f\\
 &+Z_{j+1} \mdbar N_q f + Z_{j+1}\mdbar^{\ast}N_q f +
 Z_j N_q f .
\end{align*}
 By definition
\begin{equation*}
 \square N_q f = f- H_qf,
\end{equation*}
and so we can write
\begin{align*}
 \gamma^{3j} N_q f =&
 \gamma^{\ast} (\gamma^{3j-1} f, \lrn_q)
 -\gamma^{\ast}(\gamma^{3j-1}H_q f,\lrn_q)
 +\sum_{k=3}^{j+1} Z_k\gamma^{3(j-k)+5}(f-H_qf)\\
 &+Z_{j+1} \mdbar N_q f + Z_{j+1}\mdbar^{\ast}N_q f +
 Z_j N_q f\\
=&
 \gamma^{\ast} (\gamma^{3j-1} f, \lrn_q)
 + Z_3\gamma^{2}f+Z_j H_qf\\
&
 +Z_{j+1} \mdbar N_q f + Z_{j+1}\mdbar^{\ast}N_q f +
 Z_j N_q f ,
\end{align*}
where we use
\begin{equation*}
\gamma^{\ast}(\gamma^{3j-1}H_q f,\lrn_q)
 = Z_j H_q f
\end{equation*}
and
\begin{equation*}
Z_k\gamma^{3(j-k)+5}H_qf = Z_k \circ Z_{j-k}H_qf = Z_j H_qf.
\end{equation*}

$ii.$  To prove $ii.$ we note the first $Z_1$ operator on the
right hand side of Theorem \ref{mainint} $ii)$, in view of Theorem
\ref{nkern},
 is related to
$\lrt_{q-1}$ by
\begin{equation*}
Z_1\gamma=\gamma^{\ast}\lrt_{q-1}\gamma
 +Z_2.
\end{equation*}
We therefore write Theorem \ref{mainint} $ii)$ as
\begin{equation*}
\gamma^{3j}\mdbar^{\ast} f
 =\gamma^{\ast}(\gamma^{3j-1} \square f, \lrt_{q-1})
+
 \sum_{k=2}^{j} Z_k\gamma^{3(j-k)+2}\square f
 +Z_{j} \mdbar f + Z_{j}\mdbar^{\ast}f.
\end{equation*}
Replacing $f$ with $N_qf$, we obtain
\begin{align*}
\gamma^{3j}\mdbar^{\ast} N_q f
 =& \gamma^{\ast}(\gamma^{3j-1} f, \lrt_{q-1})
 -\gamma^{\ast}(\gamma^{3j-1} H_q f, \lrt_{q-1})
\\ & +Z_2\gamma^2 f-
 \sum_{k=2}^{j} Z_k\gamma^{3(j-k)+2} H_q f
+Z_{j} \mdbar N_q f + Z_{j}\mdbar^{\ast}N_q f \\
=& \gamma^{\ast}(\gamma^{3j-1} f, \lrt_{q-1}) +
  Z_2\gamma^{3(j-k)+2} f +Z_j H_q f
  +Z_{j} \mdbar N_q f + Z_{j}\mdbar^{\ast}N_q f .
\end{align*}

$iii)$  The proof of $iii)$, in which we make use of the relation
Theorem \ref{mainint} $iii)$, follows as does that of $ii)$
\end{proof}

As mentioned above, in $\mathbb{C}^n$ we would have $H=0$, and the
proof simplifies. In particular, a cruder version of Theorem
\ref{mainint} would suffice.

Our Main Theorems \ref{mainpp} and \ref{mainnq} now follow
 from Theorem \ref{intmain} after taking into account
Lemma \ref{gamcom} and the types of the various operators.  For
instance, setting
\begin{equation*}
N_q^0 f = (f,\lrn_q),
\end{equation*}
we have
\begin{equation*}
\gamma^{3j} N_q f = \gamma^{3j} N_q^0 f +Z_3 \gamma^2 f
 +C_j^{(j)}f.
\end{equation*}

  In
particular, we can add to Main Theorem \ref{mainpp}
\begin{thrm} For $1\le q\le n-3$,
\begin{align*}
 &i)\ N_q \overset{pp}{=} N_q^0 \mbox{ of type }2\\
 &ii)\ \mdbar  N_q \overset{pp}{=}  {\bold T}_q^{\ast} \mbox{ of type }1\\
 &iii)\ \mdbar^{\ast}  N_q \overset{pp}{=}  {\bold T}_{q-1} \mbox{ of type }1.
 \end{align*}
\end{thrm}

We note that in the smooth case the above theorems coincide with
the known results (see \cite{LiMi}):  just set $\gamma\equiv 1$.

\end{document}